\documentclass[11pt]{article}
\usepackage{a4, amsthm, amsmath, amssymb,amsbsy, epic,graphicx,mathrsfs,enumerate, color}
\usepackage[all]{xy}
\usepackage{multicol,longtable}
\usepackage{setspace}

\newtheorem{thm}{Theorem}[section]
\newtheorem{lem}[thm]{Lemma}
\newtheorem{cor}[thm]{Corollary}
\newtheorem{prop}[thm]{Proposition}
\newtheorem{defi}{Definition}
\newtheorem{conj}{Conjecture}

\newtheorem{hyp}[thm]{Hypothesis}

\theoremstyle{definition}
\newtheorem*{remark}{Remark}
\newtheorem{exam}{Example}[section]

\renewcommand{\ni}{\noindent}

\renewcommand{\ker}{\operatorname{ker}}               

\newcommand{\im}{\operatorname{im}}

\newcommand{\ra}{\rightarrow}
\newcommand{\mZ}{\mathbb{Z}}
\newcommand{\mF}{\mathbb{F}}

\newcommand{\mC}{\mathbb{C}}
\newcommand{\mR}{\mathbb{R}}
\newcommand{\mN}{\mathbb{N}}

\newcommand{\Ann}{\operatorname{Ann}}
\newcommand{\Stab}{\operatorname{Stab}}

\renewcommand{\rm}{\mathrm}

\newcommand{\mbf}{\mathbf}
\newcommand{\Cal}{\mathcal}
\newcommand{\End}{\operatorname{End}}

\newcommand{\thra}{\twoheadrightarrow}

\newcommand{\sm}{\setminus}

\newcommand{\lda}{\lambda}

\newcommand{\Irr}{\operatorname{Irr}}
\renewcommand{\a}{\alpha}
\renewcommand{\b}{\beta}
\renewcommand{\k}{\kappa}
\renewcommand{\d}{\delta}
\renewcommand{\l}{\lambda}
\newcommand{\om}{\omega}
\renewcommand{\th}{\theta}
\newcommand{\lan}{\langle}
\newcommand{\ran}{\rangle}
\newcommand{\Top}{\operatorname{Top}}
\newcommand{\Supp}{\operatorname{Supp}}
\newcommand{\Infl}{\operatorname{Infl}}
\newcommand{\Defl}{\operatorname{Defl}}
\newcommand{\trace}{\operatorname{trace}}

\title{Reduction for characters of finite algebra groups}
\author{Anton Evseev\thanks{{\em E-mail address:} A.Evseev@qmul.ac.uk} \\
\bigskip
{\em Queen Mary, University of London}}
\date{}

\begin{document}

\maketitle

\abstract{Let $J$ be a finite-dimensional nilpotent algebra over a finite field $\mF_q$.
We formulate a procedure for analysing characters of the group $1+J$. In particular, we study characters of the group $U_n (q)$ of unipotent triangular 
$n\times n$ matrices over $\mF_q$. 
Using our procedure, we compute the number of irreducible characters of $U_n (q)$ of each degree for $n\le 13$. Also, we explain and generalise a phenomenon concerning
the group $U_{13}(2)$ discovered by Isaacs and Karagueuzian.

\medskip
\ni
}


\section{Introduction}\label{parintro}

In this paper we analyse complex characters of a special kind of $p$-groups.
Let $J$ be a finite-dimensional nilpotent algebra over a finite field $\mF_q$, where $q$ is a prime power. The group 
$$1+J=\{1+x \! : \; x\in J\},$$ 
with the multiplication law
$$
(1+x)(1+y) = 1+x + y + xy,
$$
is called an $\mF_q$-\emph{algebra group} (see~\cite{Isaacs1995}).
An important example is the group $U_n (q)$ of unipotent upper-triangular matrices over $\mF_q$. Clearly, $U_n (q)=1+T_n (q)$ where $T_n (q)$ is the algebra of nilpotent upper-triangular matrices over $\mF_q$, so $U_n (q)$ is indeed an algebra group. Understanding conjugacy classes and complex irreducible characters of $U_n (q)$ has proved to be a hard problem. In particular, it is a long-standing conjecture of Higman~\cite{Higman1} that the number of conjugacy classes of $U_n (q)$ is, for fixed $n$, polynomial in $q$ with integer coefficients. 

\begin{thm}[Isaacs \cite{Isaacs1995}]\label{powerq} Let $1+J$ be an $\mF_q$-algebra group. Then the degree of each complex irreducible character of $1+J$ is a power of $q$.
\end{thm}

Let $N_{n,e}(q)$ be the number of irreducible characters of  $U_n (q)$ of degree $q^e$. Theorem~\ref{powerq} shows that every irreducible character of $U_n (q)$ is counted 
in $N_{n,e}(q)$ for some $e$. The conjecture of Higman has been refined by Lehrer and Isaacs as follows.

\begin{conj}[\cite{Isaacs2007}]\label{conj_poly} For fixed $n$ and $e$, $N_{n,e}(q)$ can be expressed as a polynomial in $q$ with integer coefficients.
\end{conj}

If $q$ is a power of $2$, one can distinguish three (disjoint) types of characters of $\mF_q$-algebra groups: characters afforded by representations realisable over $\mR$ (or \emph{characters of real type}, for short), characters that are real-valued but are not of real type, and characters that are not real-valued. Isaacs and Karagueuzian~\cite{IK2005} studied characters of the groups $U_n (2)$ and discovered the following interesting facts.

\begin{thm}[Isaacs--Karagueuzian \cite{IK2005}] If $n\le 12$, every representation of $U_n (2)$ is realisable over $\mR$. However, if $n\ge 13$, there exists an irreducible character of $U_n (2)$ which is not real-valued.
\end{thm}

Marberg~\cite{Marberg2008} found an explicit description of the irreducible characters of $U_{13}(2)$ that are not real-valued. In particular, he proved the following result, 
conjectured in~\cite{IK2005}.

\begin{thm}[\cite{Marberg2008}, Theorem 9.2]\label{Marberg}  There is precisely one complex conjugate pair $\{\chi,\bar\chi\}$ of irreducible characters of $U_{13}(2)$ which 
are not real-valued, and $\chi(1)=2^{16}$. 
All the other characters of $U_{13}(2)$ are of real type.
\end{thm}

We will give an independent proof of this result, using different methods, and will show that essentially the same pattern holds for arbitrary $q$. If $q$ is odd, 
no algebra group $1+J$ defined over $\mF_q$ has a non-trivial real-valued irreducible character. However, another concept captures similar structural properties of 
characters for arbitrary $q$. As usual, we denote by $\Irr(G)$ the set of all complex irreducible characters of a finite group $G$.

\begin{defi}\label{widef} Let $1+J$ be an algebra group. We say that $\chi\in \Irr(1+J)$ is \emph{well-induced} if there exist a subalgebra $K$ of $J$ and a linear character $\phi$ of $K$ such that $\phi^{1+J}=\chi$ and $1+K^2 \subseteq \ker\phi$.
\end{defi}

\begin{remark}
By a theorem of Halasi~\cite{Halasi2006}, every irreducible character of an algebra group $1+J$ is induced from a linear character of some subgroup $1+K$ where
$K$ is a subalgebra of $J$. 
\end{remark}

Clearly, if $q$ is a power of $2$ then every well-induced $\chi\in \Irr(1+J)$ is of real type: if $K$ and $\phi$ are as in Definition~\ref{widef}, then $(1+K)/(1+K^2)$ is elementary abelian, so $\phi$ is of real type, whence $\chi$ is too. The converse may not be true in general; however, it happens to hold
for the groups $U_{13}(2^d)$. We now state one of the main results of the present paper, which implies Theorem~\ref{Marberg}.

\begin{thm}\label{wi} Let $q$ be a prime power. If $n\le 12$, then all irreducible characters of $\Irr(U_n (q))$ are well-induced. However, there are precisely 
$q(q-1)^{13}$ characters of $U_{13}(q)$ which are not well-induced, and all such characters have degree $q^{16}$. If $q$ is a power of $2$ then 
none of these $q(q-1)^{13}$ characters is real-valued.
\end{thm}

We will also prove that a real-valued character of a unitriangular group need not be of real type, 
answering a question of Isaacs and Karagueuzian (\cite{IK2005}, Problem 1.3).

\begin{thm}\label{size25} If $q$ is even, there exists a character of $U_{25}(q)$ that is real-valued but is not of real type.
\end{thm}

In addition, we will describe an algorithm computing the values of $N_{n,e}(q)$ for $n\le 13$ and arbitrary $e$ and $q$ and will observe that Conjecture~\ref{conj_poly} holds for $n\le 13$. This extends a calculation of Isaacs~\cite{Isaacs2007} for $n\le 9$. We note that the total number of conjugacy classes of $U_n (q)$, $n\le 13$, has been computed by Arregi and Vera-Lopez, see~\cite{pol2003}.

This computation and the proof of the results above are based on certain operations, which we call~\emph{contractions}. Each contraction
replaces a pair $(J,\chi)$, where $\chi\in\Irr(1+J)$, with another such pair $(J',\chi')$ according to certain rules; in particular, $\dim J'<\dim J$. 
If $(J,\chi)$ is replaced with a pair $(K,\phi)$ 
by a series of such steps and $(K,\phi)$ cannot be contracted any further then we will call $(K,\phi)$ a \emph{core} of $(J,\chi)$. Such a core seems to capture, 
to some extent, structural properties of $\chi$ and the interplay between $\chi$ as a group-theoretic object and the nilpotent algebra $J$.
In particular, we will see that the core $(K,\phi)$ is essentially unique and that reducing to the core preserves properties such as being well-induced 
and real-valued.

\begin{remark} A different reduction process for characters of algebra groups has been developed by Boyarchenko~\cite{Boyarchenko2006}.
\end{remark}

The definition of contraction and proofs of results concerning contractions 
are given in Section~\ref{contractions}. At the end of that section we state a result (Theorem~\ref{main}) 
which describes the cores of all pairs $(T_n (q),\chi)$, where $\chi\in\Irr(U_n (q)$, for $n\le 13$. Theorem~\ref{wi} will follow immediately. 
The algorithm that proves Theorem~\ref{main} and computes the values of $N_{n,e}(q)$ for $n\le 13$ is expounded over
 Sections~\ref{alg},~\ref{spec_cases} and~\ref{output}. (Section~\ref{spec_cases} contains a fast-track procedure used in special cases, whereas 
Section~\ref{alg} gives a more general algorithm.) The polynomials $N_{n,e}(q)$ for $10\le n\le 13$ are given in the Appendix. Finally, in Section~\ref{examples} we give examples of certain contractions, aiming to describe explicitly where the characters of $U_{13}(q)$ which are not well-induced come from. Section~\ref{examples} also 
contains an example proving Theorem~\ref{size25}.

\bigskip

\ni\textbf{Notation.} Our notation is mostly standard. We assume all algebras to be finite-dimensional. If $Y$ is a subset of an algebra $J$, we shall write
$$
C_J (Y) = \{ x\in J \! : \; yx=xy \text{ for all } y\in Y\}.
$$
If $Z$ is an ideal in $J$ and $X$ is a vector subspace of $J$, we shall sometimes abuse notation by writing $X$ for the subspace $(X+Z)/Z$ of $J/Z$. 
If $U$ is a left $J$-module and $u\in U$, we write
$$
\Ann_J (u) = \{ x\in J \!:\; xu=0\}.
$$
If $x_1,\ldots,x_n$ are elements of a vector space
we denote by $\lan x_1,\ldots,x_n \ran$ the linear span of $x_1,\ldots,x_n$. 

Let $N$ be a normal subgroup of a finite group $G$. Each character $\th$ of $G/N$ can be inflated to a character $\chi$ of $G$: $\chi(g)=\th(gN)$ for all $g\in G$. 
We shall write $\th=\Defl_N (\chi)$ or $\th=\tilde{\chi}$, and $\chi=\Infl_{N}(\th)$. Suppose $\chi$ is a character of $G$. We use the notation
$$
Z(\chi) = \{ g\in G \!:\; \chi(g)=\chi(1) \}.
$$
Whenever $\th$ is a character of a subgroup $H$ of $G$, we denote the character of $G$ induced from $\th$ by $\th^G$ and the restriction of $\chi$ to $H$ by $\chi_H$. The inner product of two characters $\chi$ and $\chi'$ of $G$ is denoted by $\lan \chi,\chi'\ran$. If $N$ is a normal subgroup of $G$ and $\th\in\Irr(N)$, we write
$$
\Irr(G|\th) = \{ \chi\in \Irr(G) \!: \; \lan \chi_N,\th\ran > 0 \}.
$$

As usual, $\mF_q[x]$ is the algebra of polynomials in one variable over $\mF_q$. If $f\in \mF_q[x]$, then $(f)=f\mF_q[x]$ is the ideal generated by $f$ in $\mF_q[x]$. If $k$ and $l$ are integers, we use the notation
$$
[k,l] = \{ i\in \mZ \!:\; k\le i\le l \}.
$$
If $A$ and $B$ are finite sets, we denote by $M_{A,B}(q)$ the set of matrices over $\mF_q$ with rows indexed by elements of $A$ and columns, by elements of $B$. Note that, if $C$ is another finite set, there is a natural multiplication 
$$
M_{A,B} (q) \times M_{B,C}(q) \ra M_{A,C}(q).
$$

\medskip
\ni\textbf{Acknowledgements.} Much of this work was carried out during the author's stays as a Leibniz Fellow at the Mathematisches Forschungsinstitut Oberwolfach. The author is very grateful to MFO staff for their hospitality. The author would like to thank Martin Isaacs for checking some of the results detailed here against previously 
available data and Persi Diaconis for bringing his attention to Eric Marberg's thesis~\cite{Marberg2008}.

\section{Contractions}\label{contractions}

Let $J$ be a finite-dimensional nilpotent algebra and $\chi\in \Irr(1+J)$. We shall call $(J,\chi)$ an \emph{AC-pair}. 
Sometimes we will specify an AC-pair only by giving the character, provided it is clear what the algebra is.
There is a natural notion of isomorphism of AC-pairs: two AC-pairs $(J,\chi)$ and $(J',\chi')$ are said to be \emph{isomorphic} if there exists an algebra isomorphism 
$f:J \ra J'$ such that $\chi'(1+f(x))=\chi(1+x)$ for all $x\in J$. Usually, we shall consider AC-pairs only up to isomorphism. We call an AC-pair $(J,\chi)$ \emph{small} if 
$\dim J\le 1$. We shall say that $\dim J$ and $\chi(1)$ are the \emph{dimension} and the \emph{degree} of an AC-pair $(J,\chi)$, respectively.

\begin{defi}\label{goodpair} Let $J$ be a finite-dimensional nilpotent algebra over $\mF_q$. We call a pair $(Z,Y)$ of $1$-dimensional subspaces of $J$ \emph{good} if the following conditions are satisfied:
\begin{enumerate}[(i)]
    \item\label{condZ} $JZ=ZJ=0$; 
    \item\label{condY} $JY\subseteq Z$ and $YJ\subseteq Z$;
    \item\label{condC} the centraliser $C \! :=C_J (Y)$ satisfies $YC=CY=0$;    
    \item\label{condCproper} $C\ne J$.
  \end{enumerate}
\end{defi}

Let $J$ be a nilpotent algebra over $\mF_q$, and suppose $Z$ is a $1$-dimensional ideal of $J$. Throughout the paper, we shall use the notation
$$
\Irr(1+J,Z) =\{\chi\in\Irr(1+J) \!:\; 1+Z\nsubseteq \ker\chi\}.
$$

Our analysis of characters of algebra groups is largely based on the following simple result.
\begin{lem}\label{adjustment}
Let $J$ be a nilpotent algebra over $\mF_q$. Suppose $(Z,Y)$ is a good pair of $1$-dimensional subspaces in $J$ and let $C=C_J (Y)$. 
Let $\chi\in \Irr(1+J,Z)$.
 Then $C$ is an ideal of $J$ and $\chi_{1+C}$ has a unique irreducible constituent $\theta$ such that $1+Y\subseteq \ker \theta$. Moreover, $\Stab_{1+J}(\theta)=1+C$ and 
$\theta^{1+J}=\chi$.
\end{lem}
\begin{proof} Since $J^2 Y=Y J^2=0$, we have $J^2 \subseteq C$. Hence $C$ is an ideal of $J$. Since $1+Z$ is contained in the centre of $1+J$, 
we have $1+Z\subseteq Z(\chi)$. Thus $\chi_{1+Z}=m\lda$ for some non-trivial linear character $\lda$ of $1+Z$ and an integer $m>0$.  Let $\mu$ be an extension of $\lda$ to the abelian group $1+Y+Z$. We claim that $\Stab_{1+J}(\mu)=1+C$. As $1+Y+Z$ is a central subgroup of $1+C$, it is clear that $1+C$ stabilises $\mu$. On the other hand, if 
$x\in J\setminus C$, then there exists $y\in Y$ such that $[y,x]\ne 0$. Multiplying $y$ by a non-zero scalar if necessary, we can ensure that 
$1+[y,x]\notin\ker \lda$. However, using the equalities $J^2 Y=JYJ=YJ^2=0$, we deduce 
$$
\mu((1+y)^{1+x})=\mu(1+y+yx-xy)=\mu(1+y)\mu(1+[y,x]) \ne \mu (1+y),
$$
so $1+x\notin \Stab_{1+J}(\mu)$, and the claim is proved. 

We deduce that $(1+J)/(1+C)$ acts freely and transitively on the $q$ extensions of $\lda$ to $1+Y+Z$. Since $1+Y+Z$ decomposes as a direct product of $1+Y$ and $1+Z$, there exists a unique extension $\nu$ of $\lda$ to $1+Y+Z$ such that $1+Y\subseteq \ker \nu$. Let $\eta$ be an irreducible constituent of $\chi_{1+C}$. Then $\eta_{1+Y+Z}=l\k$ for some extension $\k$ of $\lda$ and an integer $l>0$. We have proved that there exists $g\in 1+J$ such that $\k^g=\nu$. 
The character $\theta \! :=\eta^g$ is a constituent of $\chi_{1+C}$ such that $\theta_{1+Y+Z}=l\mu$, and consequently 
$1+Y\subseteq \ker\theta$. Since $\Stab_{1+J}(\mu)=1+C$, we have $\Stab_{1+J}(\theta)=1+C$. If, for some $h\in 1+J$, $\theta^h$ is another irreducible constituent of $\chi_{1+C}$ with kernel containing $1+Y$ then $\mu^h=\mu$, so $h\in 1+C$ and therefore $\theta^h=\theta$.
\end{proof}

We define two types of \emph{contraction}, each of which replaces an AC-pair $(J,\chi)$ with one of a smaller dimension (the term contraction will mean both the process of such replacement and the 
result of that process):

\begin{enumerate}[(A)]
\item[\textbf{A}] Suppose $I$ is an ideal of $J$ such that $1+I\subseteq \ker\chi$. Then $(J/I,\tilde{\chi})$ is a contraction of $(J,\chi)$. (Note that
$1+I$ is a normal subgroup of $1+J$ and $1+(J/I)$ can be identified with $(1+J)/(1+I)$.)
\item[\textbf{B}] Suppose that $(Z,Y)$ is a good pair in $J$ and $1+Z\nsubseteq \ker\chi$. 
By Lemma~\ref{adjustment} there exists a unique $\theta\in \Irr(C)$ such that $\theta | \chi_{1+C}$ and $1+Y \subseteq \ker \theta$. We say that replacing $(J,\chi)$ with
$(C/Y,\tilde{\theta})$ is a contraction. We say that this contraction~\emph{employs} the good pair $(Z,Y)$.
\end{enumerate}

We say that an AC-pair $(J'\chi')$ is an \emph{offspring} of $(J,\chi)$ if it is obtained from $(J,\chi)$ by a series of contractions. 
We call $(J,\chi)$ \emph{uncontractible} if it has no offsprings of dimension smaller than $\dim J$. If $(J',\chi')$ is an uncontractible offspring of an AC-pair 
$(J,\chi)$, we say that $(J',\chi')$ is a \emph{core} of $(J,\chi)$.

\begin{remark} It is not difficult to see that, whenever $(J,\chi)$ is a non-small AC-pair, either there is a contraction of $(J,\chi)$ of Type A which reduces the dimension or one can find $1$-dimensional subspaces $Y$ and $Z$ in $J$ such that conditions (i), (ii) and (iv) of Definition~\ref{goodpair} are satisfied and $\chi\in\Irr(1+J,Z)$. Thus
condition (iii) is the key one.
\end{remark}

The following result is an immediate consequence of Lemma~\ref{adjustment}. It shows that one can use contraction to reduce classification of characters of an algebra group $1+J$ to classification of characters of smaller algebra groups provided a good pair can be found in $J$.

\begin{cor}\label{bijection}
Let $J$ be a finite-dimensional nilpotent algebra over $\mF_q$. Suppose $(Z,Y)$ is a good pair in $J$. Let $C=C_{J}(Y)$. Then contractions of Type B yield a bijection between
$\Irr(1+J,Z)$ and $\Irr(1+C/Y,(Z+Y)/Y)$.
The inverse of this map is given by $\tilde{\theta}\mapsto \theta^{1+J}$.
\end{cor}

\begin{thm}\label{cores} Any two cores of an AC-pair $(J,\chi)$ are isomorphic.
\end{thm}

The theorem will easily follow once we prove the following result. 

\begin{lem}\label{cores_aux} 
 Suppose $\mbf R_1$ and $\mbf R_2$ are both contractions of an AC-pair $(J,\chi)$. Then $\mbf R_1$ and $\mbf R_2$ have a common offspring.
\end{lem}

\begin{proof} 
\textbf{Case 1:} both contractions in the hypothesis are of Type A. 
\nopagebreak

Then $\mbf R_1 = (J/I_1,\Defl_{1+I_1}(\chi))$ and $\mbf R_2 = (J/I_2,\Defl_{1+I_2}(\chi))$ where $I_1$ and $I_2$ are ideals of $J$ with $1+I_1+I_2\subseteq \ker \chi$. Clearly, 
$(J/(I_1+I_2),\Defl_{1+I_1+I_2}(\chi))$ is a common offspring of $\mbf R_1$ and $\mbf R_2$.

\smallskip
\ni\textbf{Case 2:} exactly one of the two contractions is of Type B, say, the contraction to $\mbf R_1$. 
\nopagebreak

Let $(Z,Y)$ be the good pair employed in that contraction and let 
$C=C_J (Y)$. Then $\mbf R_1 = (C/Y,\tilde{\theta})$ for a certain constituent $\theta$ of $\chi_{1+C}$, and $\mbf R_2=(J/I,\tilde{\chi})$, where $I$ is an ideal of $J$. 
 Since $[J,Y] \subseteq Z$, we have $[I,Y]\subseteq I\cap Z$. Since $1+I \subseteq \ker \chi$ but $1+Z \not\subseteq \ker\chi$, we have $I\cap Z=0$, so $[I,Y]=0$, and therefore $I\subseteq C$. Since $1+I\subseteq \ker\chi$ and $\th$ is a constituent of $\chi_{1+C}$, we see that $1+I\subseteq\ker\th$. Let $\phi\in\Irr(1+C/(I+Y))$ be the deflation of $\th$.
Then $(C/(I+Y),\phi)$ is a contraction of $(C/Y,\tilde{\th})$ of Type A. 

To complete the proof in this case, we show that $(C/(I+Y),\phi)$ can be obtained via a contraction of $\mbf R_2$ of Type B 
employing the good pair $((Y+I)/I,(Z+I)/I)$. This will be clear once we show that $C'=C$ where 
$C'=\{ x\in J \!:\, [x,Y]\subseteq I\}$. We certainly have $C\subseteq C'$. Since 
$\dim(J/C)=1$, it suffices to show that $C'\ne J$. However, since $C\ne J$, there exist $x\in J$ and $y\in Y$ such that $[x,y]\in Z\setminus\{0\}$. 
As $Z\cap I=0$, we have $[x,y]\notin I$, so $C'\ne J$, as required.

\smallskip
\ni\textbf{Case 3:} both contractions are of Type B. 
\nopagebreak

Suppose the contraction to $\mbf R_1$ employs the good pair $(Z_1,Y_1)$, and the contraction to $\mbf R_2$ employs $(Z_2,Y_2)$.
 For $i=1,2$, let $C_i=C_J (Y_i)$. Then $\mbf R_i=(C_i/Y_i,\tilde{\th}_i)$ where $\th_i$ 
is an irreducible constituent of $\chi_{1+C_i}$. We distinguish three subcases.

\smallskip
\ni \textbf{Subcase 3a:} $Y_1 Y_2=Y_2 Y_1=0$ and $C_1=C_2=: \! C$. 
\nopagebreak

As $\th_1$ and $\th_2$ are both constituents of $\chi_{1+C}$, there exists $g\in 1+J$ such that $\th_2=\th_1^g$. 
Since 
$1+Y_1\subseteq \ker \th_1$, we have
 $1+Y_1^g\subseteq \ker \th_2$. Moreover, $C^g=C$ because $C$ is an ideal of $J$, so 
we have $C Y_1^g=Y_1^g C=0$. Thus, if we set $\phi_2$ to be the character of $1+C/(Y_1^g+Y_2)$ 
obtained from $\th_2$ by deflation, then $\mbf S_2=(C/(Y_1^g+Y_2),\phi_2)$ is a Type A contraction of $\mbf R_2$. 
Similarly, we obtain a Type A contraction $\mbf S_1=(C/(Y_1+Y_2^{g^{-1}}),\phi_1)$ of $\mbf R_1$, where $\phi_1=\phi_2^{g^{-1}}$. 
Conjugation by $g$ is an isomorphism between $\mbf S_1$ and $\mbf S_2$.

\smallskip
\ni\textbf{Subcase 3b:} $Y_1 Y_2 = Y_2 Y_1=0$ and $C_1\ne C_2$. 
\nopagebreak

Then clearly $Z_2\subseteq C_1$, $Y_2\subseteq C_1$ and $Y_1\subseteq C_2$. We claim that $((Z_2+Y_1)/Y_1,(Y_2+Y_1)/Y_1)$ is a good pair in $C_1/Y_1$. 
To see this, it is enough to check that 
\begin{equation}\label{eq3b}
C_1\cap C_2=\{ x\in C_1 \! : \; [x,Y_2]\subseteq Y_1\}.
\end{equation}
We have $1+Z_2\subseteq Z(\chi)$ and $1+Z_2\nsubseteq \ker\chi$, whence $1+Z_2\subseteq Z(\th_1)$ and $1+Z_2\nsubseteq\ker(\th_1)$. Therefore, $Z_2\ne Y_1$ and, 
as $[J,Y_2]\subseteq Z_2$, equality~\eqref{eq3b} follows, and with it the claim. 
We have already shown that $1+Z_2\nsubseteq\ker(\tilde{\th}_1)$, so a Type B contraction employing $(Z_2,Y_2)$ can be applied to $\mbf R_1$. The result of this contraction is 
the AC-pair $\mbf S=(C_1\cap C_2/(Y_1+Y_2),\tilde{\phi})$ where $\phi\in \Irr(1+C_1\cap C_2)$ is a constituent of $\chi_{1+C_1\cap C_2}$ and $1+Y_1+Y_2\subseteq \ker\phi$. 

We claim that $\phi$ is the unique irreducible 
constituent of $\chi_{1+C_1\cap C_2}$ with this property. Indeed, suppose that $\phi'$ is another such constituent. Then there exists
a constituent $\om\in\Irr(1+C_1)$ of $\chi_{1+C_1}$ such that $\phi'|\om_{1+C_1\cap C_2}$. Hence $1_{1+Y_1}$ is a constituent of $\om_{1+Y_1}$ and, as 
$1+Y_1$ is a normal subgroup of $1+C_1$, we see that $1+Y_1\subseteq\ker\om$. Therefore, by Lemma~\ref{adjustment}, $\om=\th_1$, whence by Lemma~\ref{adjustment} applied again $\phi'=\phi$. The uniqueness of $\phi$ shows that $\mbf S$ is a contraction not only of $\mbf R_1$, but also, by symmetry, of $\mbf R_2$.

\smallskip
\ni\textbf{Subcase 3c:} $Y_1 Y_2 + Y_2 Y_1\ne 0$. 
\nopagebreak

It follows from the definition of a good pair that $Y_2 \not\subseteq C_1$ and $Y_1\not\subseteq C_2$. Let $D=C_1\cap C_2$. Then 
$DY_1=Y_1 D=0$ and $C_1=D\oplus Y_1$ as a vector space, so the group $1+C_1$ decomposes as a direct product of $1+D$ and $1+Y_1$. 
Hence the character $\phi \! :=(\th_1)_{1+D}$ is irreducible. Define the algebra homomorphism 
$f: D \ra C_1/Y_1$ by $f(x)=x+Y_1$. By our assumption, $D\cap Y_1 = 0$, so $f$ is injective, and therefore surjective. Also, $f$ induces a group homomorphism $f:1+D \ra 1+C_1/Y_1$ in the obvious way, and $\tilde{\th}_1 \circ f = \phi$. Thus $\mbf S:=(D,\phi)$ is isomorphic to $\mbf R_1$. Similarly, $\mbf S$ is isomorphic to $\mbf R_2$, so $\mbf R_1$ and $\mbf R_2$ are isomorphic. 
\end{proof}
 
\begin{proof}[Proof of Theorem~\ref{cores}] We argue by induction on $\dim J$. Suppose $\mbf T_1$ and $\mbf T_2$ are cores of $(J,\chi)$.  Let $\mbf R_1$ and $\mbf R_2$ be the contractions of $(J,\chi)$ obtained in the first steps of sequences of contractions leading to $\mbf T_1$ and $\mbf T_2$ respectively. We may assume that $\dim R_i<\dim J$ for $i=1,2$. By Lemma~\ref{cores_aux}, $\mbf R_1$ and $\mbf R_2$ have a common offspring $\mbf S$, say.
 Let $\mbf Q$ be a core of $\mbf S$. Since $\mbf Q$ and $\mbf T_1$ are both cores of $\mbf R_1$, the AC-pairs $\mbf Q$ and $\mbf T_1$ are isomorphic by the inductive hypothesis. Similarly, $\mbf Q \cong \mbf T_2$, whence $\mbf T_1\cong \mbf T_2$. 
\end{proof}

Next we show that contractions preserve a number interesting properties of characters of algebra groups. 

\begin{prop}\label{wipres} Let $(K,\theta)$ be an offspring of an AC-pair $(J,\chi)$. Then $\theta$ is well-induced if and only if $\chi$ is well-induced. 
\end{prop}

\begin{proof}  We may assume that $(K,\th)$ is obtained from $(J,\chi)$ by just one contraction. Suppose first that this contraction is of Type A, so 
$K=J/I$ for some ideal $I$ and $\th=\tilde{\chi}$. Assume that $\chi$ is well-induced. Then there exists a subalgebra $L$ of $J$ and a linear character $\psi$ of $1+L$ such that $\psi^{1+J}=\chi$ and $1+L^2\subseteq \ker\psi$. By~\cite{IsaacsBook}, Lemma 5.11, we have $\ker \chi\subseteq \ker\psi\subseteq 1+L$, whence $1+I\subseteq \ker\psi$.  
Therefore, $\psi$ can be deflated to a character $\tilde{\psi}\in\Irr(1+L/I)$ and $\tilde{\psi}^{1+J/I}=\th$. Since $1+(L/I)^2 \subseteq \ker\tilde{\psi}$, we see that $\th$ is well-induced. The converse is clear.   

Now consider the case when the contraction is of Type B. Let $(Z,Y)$ be the good pair employed by the contraction, so $K=C_J (Y)/Y=C/Y$, and let $\xi=\Infl_{1+Y}(\th)$.
By Corollary~\ref{bijection}, $\chi=\xi^{1+J}$. Thus if $\theta$ is well-induced then so is $\chi$. 

Suppose then that $\chi$ is well-induced, that is, there exist a subalgebra $L$ of $J$ and a linear character $\phi\in\Irr(1+L)$ such that 
$\phi^{1+J}=\chi$ and $1+L^2 \subseteq \ker\phi$. By the first part of the proof, $\th$ is well-induced whenever $\xi$ is, so it suffices to show that $\xi$ is well-induced. 
First, consider the case $L\subseteq C$. Since $\phi^{1+J}$ is irreducible, so is $\eta \! :=\phi^{1+C}$. By Clifford theory there exists $g\in 1+J$ such that $\eta^g=\xi$. Then the pair $(L^g,\phi^g)$ witnesses the fact that $\xi$ is well-induced.

Finally, assume $L\not\subseteq C$. Let $\psi=\phi_{1+L\cap C}$. Let $\xi_1=\xi,\xi_2,\ldots,\xi_q$ be the irreducible constituents of 
$\chi_{1+C}$. Since $\dim(J/C)=1$, we have $LC=J$, whence $(1+L)(1+C)=1+J$ (as $C$ is an ideal of $J$). Thus by the Mackey formula 
\begin{equation}\label{psi}
\psi^{1+C}=(\phi^{1+J})_{1+K}= \xi_1+\cdots+\xi_q.
\end{equation}
If $Y\subseteq L$, then, as $1+Y$ is a central subgroup of $1+C$, we have $(\psi^{1+C})_{1+Y}=l\mu$ where $\mu$ is a linear character of $1+Y$ and 
$l\in \mN$. However, by the proof of Lemma~\ref{adjustment}, $\xi_1,\ldots,\xi_q$ have distinct restrictions to $1+Y$. Thus $Y\cap L=0$. Let $M=(L\cap C)+Y$. Since $CY=YC=0$, 
we see that $1+M$ decomposes as a direct product of $1+L\cap C$ and $1+Y$. Let $\om$ be the unique extension of $\psi$ to $1+M$ such that $1+Y\subseteq \ker \om$. Then $\om^{1+C}(1)=\psi^{1+C}(1)/q=\theta(1)$ and $1+Y\subseteq \ker(\omega^{1+K})$. 
As $\om^{1+C}$ is a constituent of $\psi^{1+K}$, we have $\om^{1+K}=\xi_i$ for some $i\in [1,q]$ by~\eqref{psi}. By Lemma~\ref{adjustment}, $\xi$ is the only one of the 
characters 
$\xi_1,\ldots,\xi_q$ whose kernel contains $1+Y$, whence $\om^{1+K}=\xi$.
Moreover, 
$$
1+M^2\subseteq 1+(L\cap C)^2\subseteq \ker\psi \subseteq \ker\om.
$$ 
Hence $\xi$ is well-induced, and so is $\th$.
\end{proof}

\begin{prop}\label{char2eqv} Let $(K,\theta)$ be an offspring of an AC-pair $(J,\chi)$ and suppose $q$ is a power of $2$. Then 
\begin{enumerate}[(i)]
\item\label{realvaleqv} the character $\chi$ is real-valued if and only if $\theta$ is;
\item\label{realiseeqv} the character $\chi$ is of real type if and only if $\theta$ is.
\end{enumerate}
\end{prop}

\begin{proof} As in the previous proof, we may assume that $(K,\theta)$ is obtained from $(J,\chi)$ by one contraction. If this contraction is of Type A, the result is clear. 
So we suppose that the contraction is of Type B and that it employs a certain good pair $(Z,Y)$, whence $K=C_J (Y)/Y=C/Y$. Let $\xi=\Infl_{1+Y}(\th)$. 
 By Corollary~\ref{bijection}, $\chi=\xi^{1+J}$. Hence $\chi$ is real-valued if $\th$ is, and $\chi$ is of real type if $\th$ is. If $\theta$ is not real-valued then $\theta\ne \bar\theta$. Therefore, by Corollary~\ref{bijection},  $\chi=\xi^{1+J}\ne \bar\xi^{1+J}=\bar\chi$, whence $\chi$ is not real-valued.

Finally, suppose $\chi$ is of real type. Let $U$ be an $\mR (1+J)$-module affording $\chi$. Let $\lda$ be the unique irreducible constituent of $\chi_{1+Z}$, and let $\mu_1,\ldots,\mu_q$ be the extensions of $\lda$ to $1+Y+Z$. By the proof of Lemma~\ref{adjustment}, for each $\mu_i$ there exists a unique irreducible constituent $\xi_i$ of $\chi_{1+K}$ such that $\mu_i$ is a summand of $(\xi_i)_{1+Y+Z}$. (We may assume that $\xi=\xi_1$.) Since the group $1+Y+Z$ is elementary abelian, the linear characters $\mu_i$ take values $\pm 1$.
Since $U$ affords $\chi$, we have $uh=\lda(h)u$ for all $h\in 1+Z$ and $u\in U$. 
Let
$$
V_i = \{u\in U \! : \; uh=\mu_i(h) u \; \text{ for all } \; h \in 1+Y+Z \}.
$$ 
Then $U=\oplus_{i=1}^q V_i$. Since $1+Y+Z$ is a central subgroup of $1+C$, each $V_i$ is an $\mR(1+C)$-submodule of $U$. Let $\eta_i$ be the character of $1+K$ afforded by $V_i$. Then $(\eta_i)_{1+Y+Z}$ is a multiple of $\mu_i$, whence $\eta_i=k\xi_i$ for some $k\in \mN$. However, since 
$\lan \xi_i,\chi_{1+K} \ran=1$, we have $\eta_i=\theta_i$ for each $i$. Thus $V_i$ affords $\xi_i$, and in particular $V_1$ affords $\xi$. 
Therefore, $\xi$ is of real type, and so $\th$ is too.
\end{proof}

In general, the core of an AC-pair can be quite large. However, the cores of characters of the unitriangular group $U_n (q)$, $n\le 13$, are not.

\begin{thm}\label{main} Let $q$ be any prime power. If $n\le 12$ then for all $\chi\in \Irr(U_n (q))$ the core of $(T_n (q),\chi)$ is small. 
There are precisely $q(q-1)^{13}$ characters $\chi\in U_{13}(q)$ which do not have a small core. All such characters $\chi$ are of degree $q^{16}$ and have cores 
 the form $(K,\phi)$ where $K$ is the $2$-dimensional commutative algebra $x\mF_q[x]/(x^3)$ and $1+K^2\nsubseteq \ker\phi$.
\end{thm}

We show that Theorem~\ref{wi} follows from this result.

\begin{proof}[Deduction of Theorem~\ref{wi}]
By Propositions~\ref{wipres} and~\ref{char2eqv}, whenever $(T_n (q),\chi)$ has a small core, 
$\chi$ is well-induced and, if $q$ is even, is of real type. 

It remains only to consider characters $\chi\in\Irr(U_{13}(q))$ which have a $2$-dimensional core $(K,\phi)$ of the form given in Theorem~\ref{main}. Since
$\phi$ is linear and $1+K^2\nsubseteq \ker\phi$, we see that $\phi$ is not well-induced, so neither is $\chi$ (by Proposition~\ref{wipres}).
Now suppose $q$ is a power of $2$. 
 Let $z\in K^2$ be such that $1+z\notin\ker\phi$, whence
$\phi(1+z)=-1$. Writing $x$ for $x+(x^3)$, we have $x^2=\a z$ for some $\a\in \mF_q$. As $q$ is a power of $2$, there exists $\b\in\mF_q$ such that $\b^2=\a$. Then $(\b x)^2=z$, so $\phi(\b x)=\pm i$ is not 
real. Since $\phi$ is not real-valued, neither is $\chi$, by Proposition~\ref{char2eqv}.
\end{proof}

Theorem~\ref{main} is proved by a computer calculation described in the next three sections.

\section{The general algorithm}\label{alg}

We now detail an algorithm that analyses characters of algebra groups. The algorithm has been implemented in MAGMA~\cite{MAGMA}. 
None of the group-theoretic capabilities of MAGMA were used, however, so an implementation in any other language would have been essentially similar. 
Our description omits certain shortcuts used to make the program run faster. Also, some types of data are represented in the program in a different way than described here.

We say that an \emph{assembly} is a map $S$ which associates to each prime power $q$ a collection $S(q)$ of AC-pairs defined over $\mF_q$. We will deal with assemblies
that can be encoded combinatorially. We will call two assemblies $S_1$ and $S_2$ \emph{isomorphic} if for each $q$ there is a bijection from $S_1 (q)$ onto $S_2 (q)$ which maps
each AC-pair in $S_1(q)$ to an isomorphic pair.

We describe a way of encoding certain assemblies. Consider data $\mbf A$ consisting of  
\begin{enumerate}[(i)]
\item a set $Q$ of parameters. (The elements of $Q$ will be assumed to vary over $\mF_q$ whenever $q$ is specified.)
\item a set $E$ of restrictions on parameters from $Q$. These restrictions can be of the following forms:
\begin{enumerate}[(a)]
\item inequations $a\ne 0$, where $a\in \mF_q$;
\item multivariate polynomial equations with integer coefficients, with elements of $Q$ acting as variables;
\end{enumerate}
\item a (possibly, empty) linearly ordered ``basis'' set $B$;
\item a map $R:B\times B\times B \ra \{ 0 \} \sqcup \Cal P (Q)$, where $\Cal P (Q)$ denotes the power set of $Q$.
(This map encodes the structure constants of the algebras as explained below.)
\end{enumerate}

Suppose such data $\mbf A$ are given and a prime power $q$ is specified. We write $\Cal V(Q,E,q)$ for the set of all
substitutions $Q\ra \mF_q$ which make the restrictions $E$ hold (naturally, when polynomial equations are interpreted over $\mF_q$, we use the standard map $\mZ \ra \mF_q$ to evaluate the coefficients of polynomials). To each $h\in \Cal V(Q,E,q)$ we associate an (a priori, non-associative) algebra $J(\mbf A,h)$ over $\mF_q$ given by the
 basis $B$ and 
 relations
$$ 
xy=\sum_{\substack{z\in B \\ R(x,y,z)\ne 0} } \left( \prod_{a\in R(x,y,z)} h(a) \right)\! z, \quad x,y\in B.
$$
 We define 
$$
\Cal J(\mbf A,q)= \{ J(\mbf A,h) \! : \; h \in \Cal V(Q,E,q) \},
$$
where the algebras $J(\mbf A,h)$ are assumed to be distinct for different $h$.
We say that the algebras \emph{encoded} by $\mbf A$ are all the algebras $J(\mbf A,h)$ where $h$ runs over $\Cal V(Q,E,q)$ and $q$ runs over all prime powers.

Throughout, we assume that the data $\mbf A$ are such that for all $q$ all the algebras in $\Cal J(\mbf A,q)$ are associative. We also suppose 
that the basis $B$ is ordered so that,
 whenever $R(x,y,z)\ne 0$, $z$ comes after both $x$ and $y$ (so in particular the algebras encoded by $\mbf A$ are nilpotent).
We shall refer to a tuple $\mbf A$ satisfying these condition as \emph{algebraic data}.

If $\mbf A$ is algebraic data we denote by $\Irr(\mbf A)$ the assembly given by
$$
\Irr(\mbf A)(q) = \{ (J(\mbf A,h),\chi) \! : \; h\in \Cal V(Q,E,q) \text{ and } \chi\in\Irr(1+J(\mbf A,h)) \}.
$$
If moreover $z\in B$, where $B$ is the basis set of $\mbf A$, we write $\Irr(\mbf A,z)$ for the assembly given by
$$
\Irr(\mbf A,z)(q) = \{ (J,\chi)\in \Irr(\mbf A)(q) \! : \; \chi\in \Irr(1+J,\lan z\ran) \}. 
$$

Now we describe the output of the algorithm. Informally, given algebraic data $\mbf A$, the aim is to represent $\Irr(\mbf A)(q)$ (for each $q$)
as a disjoint union of certain families, with the families described independently of $q$. 
The 0-th family will consist only of AC-pairs with small cores, 
and the algorithm will attempt to find polynomials in $q$ giving the number of characters of degree $q^e$ in that family for each $e\in \mZ_{\ge 0}$. The other families 
represent cases of AC-pairs for which the algorithm is unable to prove that the core is small: in each of these cases, the algorithm returns data showing how far 
it has been able to perform contractions.

In order to describe the output precisely, we will need certain technical concepts. We  say that a \emph{categorisation} is a list 
$O=(O_0,O_1,\ldots,O_r)$ where
\begin{enumerate}[(i)]
\item $O_0 = (f,U_1,\ldots,U_n)$ and
\begin{enumerate}[(a)]
\item $f\in \mZ[q,t]$ is a formal polynomial in two variables (it represents the numbers of characters of various degrees);
\item for $1\le i\le n$, $U_i=(Q_i,E_i,u_i,v_i,e_i)$ where $Q_i$ and $E_i$ are a set of parameters and a set of restrictions for those parameters as in the definition of
algebraic data above, and 
$u_i,v_i,e_i\in \mZ_{\ge 0}$ (these data represent collections of characters that the algorithm is unable to count); 
\end{enumerate}
\item for $1\le i\le r$, $O_i = (F_i,k_i,l_i,m_i)$ where $F_i$ is an assembly and $k_i,l_i,m_i\in \mZ_{\ge 0}$. (These data account for families 
of characters which the algorithm is unable to contract to a $1$-dimensional core.)
\end{enumerate}
Whenever a categorisation will be processed by the algorithm, the assemblies $F_i$ will be given by combinatorial data. The nature of these data will vary, but it 
will be clear what the data are in each case: for example, an assembly of the form $\Irr(\mbf A,z)$ can be described by the pair $(\mbf A,z)$.

If $(J,\chi)$ is an AC-pair defined over a finite field $\mF_q$, write $d(\chi)=d$ where $\chi(1)=q^d$ (cf.\ Theorem~\ref{powerq}). 
If $X$ is a collection of AC-pairs over a fixed finite field $\mF_q$, we shall encode the numbers of AC-pairs of various degrees in $X$ by the following polynomial:
$$
\gamma_X = \sum_{(J,\chi)\in X} t^{d(\chi)} \in \mZ[t].
$$

We shall call the categorisation $O$ \emph{correct} for an assembly $S$ if for each prime power $q$ there is a decomposition
$$
S(q)= S_0 (q) \sqcup S_1 (q) \sqcup \cdots \sqcup S_r (q)
$$
and, for $1\le i\le r$, there is a map $\Phi_i (q): S_i (q)\ra F_i (q)$ satisfying the following properties:
\begin{enumerate}[(i)]
\item All AC-pairs in $S_0 (q)$ have small cores, and
$$
\gamma_{S_0 (q)} = f(q,t) + \sum_{i=1}^n (q-1)^{u_i} q^{v_i} |\Cal V(Q_i,E_i, q)| t^{e_i}.
$$  
\item  $|\Phi_i(q)^{-1} (\{ (K,\phi) \})|=(q-1)^{k_i} q^{l_i}$ for all $(K,\phi) \in F_i (q)$.
\item Whenever $\Phi_i (q)$ sends $(J,\chi)\in S_i (q)$ to $(K,\phi)\in F_i (q)$, we have $\chi(1)=q^{m_i} \phi(1)$, and $(K,\phi)$ and $(J,\chi)$ have cores that are
 either isomorphic or are both small.
\end{enumerate}

Suppose $O^{(1)},\ldots,O^{(s)}$ are categorisations, with $O^{(i)}=(O_1^{(i)},\ldots,O_{r_i}^{(i)})$. Write 
$$
O^{(i)}_0 = (f^{(i)}, U^{(i)}_1,\ldots, U^{(i)}_{n_i}).
$$
We say that the \emph{aggregate} of $O^{(1)},\ldots,O^{(c)}$ is the 
categorisation 
$$
O = (O_0, O^{(1)}_1,\ldots,O_{r_1}^{(1)}, \ldots, O^{(s)}_1,\ldots,O^{(s)}_{r_s})
$$
where
$$
O_0 = (f^{(1)}+\cdots+f^{(s)}, U_1^{(1)}, \ldots, U^{(1)}_{n_1},\ldots, U^{(s)}_1,\ldots,U^{(s)}_{n_s}).
$$

If $O$ is a categorisation and the notation above is used for its structure then we shall refer to the expressions $f,q^{u_i}(q-1)^{v_i}$, $q^{k_j}(q-1)^{l_j}$ as 
\emph{character counts} of $O$ and to the expressions $f$, $t^{e_i}$, $t^{m_j}$ as \emph{character degrees} of $O$. 
Our recursive algorithm will rely on the following straightforward result. 

\begin{lem}\label{aggr} Let $S$ be an assembly and let $O$ be a correct categorisation for $S$, with the structure of $O$ denoted as above. Suppose $O^{(i)}$ is a 
categorisation for the assembly $F_i$ for $1\le i\le r$. 
For each $i$, let $\tilde{O}^{(i)}$ be the categorisation obtained from $O^{(i)}$ by multiplying all the character counts  by 
$(q-1)^{k_i}q^{l_i}$ and all the character degrees by $t^{m_i}$. Let $M$ be the aggregate of $(O_0), \tilde{O}^{(1)},\ldots,\tilde{O}^{(r)}$. Then
$M$ is a correct categorisation for $S$.  
\end{lem}

\begin{proof} Let $S_0,S_1,\ldots,S_r$ and $\Phi_1,\ldots,\Phi_s$ be the assemblies and maps witnessing the correctness of $O$. For $1\le i\le r$, let 
$S^{(i)}_0,S^{(i)}_1,\ldots,S^{(i)}_{r_i}$ and $\Phi^{(i)}_1,\ldots,\Phi^{(i)}_{r_i}$ be the assemblies and maps witnessing the correctness of $O^{(i)}$. Let $\hat S$ be the 
assembly given by
$$
\hat S(q) = S_0 (q) \sqcup \bigsqcup_{i=1}^r S^{(i)}_0 (q).
$$
Then 
$$
S(q) = \hat S(q) \sqcup \bigsqcup_{i=1}^r \bigsqcup_{j=1}^{r_i} S^{(i)}_j (q).
$$
It is easy to check that the assemblies
$$
\hat S(q),\; S^{(i)}_j (q), \quad j=1,\ldots, r_i, \quad i=1,\ldots, r
$$
together with the maps
$$
\Phi^{(i)}_j (q) \circ \Phi_i (q), \quad j=1,\ldots, r_i, \quad i=1,\ldots, r
$$
witness the correctness of $M$. 
\end{proof}

The algorithm contains two main functions: \emph{General} and \emph{TypeB}. The function \emph{General} takes algebraic data $\mbf A$ as input
and returns a categorisation of $\Irr(\mbf A)$. Note that this specification can be satisfied rather trivially: 
the data $(O_0,O_1)$, with $O_0=(0)$ and $O_1=(\Irr(\mbf A),0,0,0)$, is a correct categorisation of 
$\Irr(\mbf A)$. In practice, \emph{General} produces much more informative categorisations, at least for algebras such as $T_n (q)$.

The function \emph{TypeB} (so named because a contraction of Type B is one of its main steps) takes a pair $(\mbf A,z)$ as input, 
where $\mbf A=(Q,E,B,R)$ is algebraic data and $z\in B$. Its output is a categorisation of $\Irr(\mbf A,z)$. 
 
The functions \emph{General} and \emph{TypeB} require the algebraic data $\mbf A=(Q,E,B,R)$ given as input to satisfy the following technical condition:
\begin{equation}\label{NZ}
 \text{for all } x,y,z\in B\!: \quad \text{if } a\in R(x,y,z) \text{ then } E \text{ contains the inequation } a \ne 0.
\end{equation}
(Here and in the sequel a statement such as $a\in R(x,y,z)$ automatically implies that $R(x,y,z)\ne 0$.)

Condition~\eqref{NZ} does not pose any substantial restriction: given any algebraic data, we can split it into cases according to whether 
relevant parameters are equal to zero or not; then we process each of those cases separately. This is done by means of the following routine. 

\medskip

\begin{samepage}
\begin{center} \textbf{Function} \emph{SplitIntoCases} \end{center}
\nopagebreak

\ni \emph{Input:} algebraic data $\mbf A=(Q,E,B,R)$.
\nopagebreak

\ni \emph{Output:} a list of algebraic data $\mbf A_1,\ldots,\mbf A_r$, each satisfying~\eqref{NZ}, such that 
$\Irr(\mbf A)(q)=\sqcup_{i=1}^r \Irr(\mbf A_i)(q)$ for all $q$.
\end{samepage}

If $\mbf A$ satisfies~\eqref{NZ} then just return $\mbf A$.
Otherwise, let $a\in Q$ be a parameter witnessing the failure of~\eqref{NZ} and proceed as follows:
\begin{enumerate}
\item Form data $\mbf A_1$ from $\mbf A$ by adding the inequation $a\ne 0$ to $E$.
\item Form data $\mbf A_2$ from $\mbf A$ by ``setting'' $a=0$. That is, whenever $a\in R(u,v,w)$ (for some $u,v,w\in B$), set $R(u,v,w)=0$. Remove all monomials involving 
$a$ from equations of $E$ and remove $a$ from $Q$. 
\item Run \emph{SplitIntoCases} recursively on $\mbf A_1$ and $\mbf A_2$, concatenate the outputs and \textbf{return} the resulting list.
\end{enumerate}
\bigskip

We now describe the main two functions, which call each other recursively. It will be clear from our description that the process always terminates.
Naturally, every \textbf{return} statement terminates the execution of the relevant function.
 Each such statement is accompanied by an explanation why the categorisation returned is correct for the input (where this is not obvious) as long as all recursive calls of 
\emph{General} and \emph{TypeB} return output that is correct for their input, as we may assume inductively.

\begin{samepage}
\medskip
\begin{center} \textbf{Function} \emph{General} \end{center}
\nopagebreak

\ni\emph{Input:} algebraic data $\mbf A=(Q,E,B,R)$ satisfying~\eqref{NZ}.
\nopagebreak

\ni\emph{Output:} a categorisation of $\mbf A$.
\end{samepage}

If $R(u,v,w)=0$ for all $u,v,w\in B$ then form the output $O=(O_0)$ as follows.
 Attempt to express $|\Cal V(Q,E,q)|$ as a polynomial in $f=f(q)\in \mZ[q]$. (We choose not to describe methods used to do this: these are diverse, but not sophisticated.)  
If the attempt is successful set $O_0=(f)$ (note that $f$ can be viewed as an element of $\mZ[q,t]$).
If not, set  $O=((0),U_1)$ with $U_1=(Q,E,0,0,0)$. \textbf{Return} $O$.

\smallskip
If $R$ is not identically zero then
\begin{enumerate}
\item Pick a vector $z\in B$ such that $R(u,z,v)=R(z,u,v)=0$ for all $u,v\in B$. For example, the last vector in the ordering of $B$ satisfies this.
 In general, there are several possibilities for $z$: 
we use an \emph{ad hoc} method of choosing $z$ to try to maximise the efficiency of the algorithm. We shall not describe this method.
\item Form algebraic data $\mbf A'$ obtained from $\mbf A$ by removing $z$ from $B$ (and restricting $R$ to $(B\sm \{z\})^3$).
\item\label{Gen1} Run \emph{General} on input $\mbf A'$, denoting the output by $O_1$.
\item\label{TB1} Run \emph{TypeB} on input $(\mbf A,z)$ and write $O_2$ for the output.
\item \textbf{Return} the aggregate $O$ of $O_1$ and $O_2$.
The categorisation $O$ is indeed correct for $\Irr(\mbf A)$: this is witnessed by the decomposition $\Irr(\mbf A)(q)=S_1(q)\sqcup S_2 (q)$, for each $q$, where
$$
S_1 (q)  =  \{ (J,\chi)\in \Irr(\mbf A)(q) \! : \; 1+\lan z\ran \subseteq \ker \chi \}
$$
and $S_2 (q)=\Irr(\mbf A,z)(q)$. 
\end{enumerate}

\bigskip

Function~\emph{TypeB} relies on two preliminary results. The first one is standard (see e.g.~\cite{Isaacs2007}, Section 3). Given a prime power $q$, we fix, 
once and for all, a non-trivial group character $\nu:(\mF_q,+)\ra \mC^{\times}$.

\begin{lem}\label{abchar} Let $U$ be a vector space over $\mF_q$. Then there is a bijection between the dual space $U^*$ and $\Irr(U)$, with $f\in U^*$ mapped to the 
character $x\mapsto \nu(f(x))$, $x\in U$. 
\end{lem}

Note that, whenever an algebra $J$ is a direct sum of ideals $J_1$ and $J_2$, we have
$1+J= (1+J_1)\times (1+J_2)$, and so every character $\chi\in\Irr(1+J)$ is of the form $\chi=\chi_1\times \chi_2$, where 
$(\chi_1\times \chi_2) (g_1g_2)=\chi_1 (g_1)\chi_2(g_2)$, $g_1\in 1+J_1$, $g_2\in 1+J_2$. 

\begin{lem}\label{nodiff} Let $(M,\th)$ and $(R\oplus M, \tau\times\th)$ be AC-pairs with $\dim R=1$ and $RM=MR=R^2=0$. 
If $\th$ is non-trivial then these two pairs have isomorphic cores. 
\end{lem}

\begin{proof} 
Let $(D,\psi)$ be the core of $(M,\th)$. It is easy to check that then $(R\oplus D,\tau\times\psi)$ is an offspring of $(R\oplus M,\tau\times \th)$:
 contractions can be performed to one summand of a direct sum of algebras without affecting the other summand. 

It is clear that any offspring of a non-trivial character is non-trivial. As $\th$ is non-trivial, so is $\psi$. 
Let $Z$ be a $1$-dimensional subspace of $D$ such that $DZ=ZD=0$. 
Then $1+Z\nsubseteq\ker\psi$: else we could apply a contraction of Type A to $(D,\psi)$ to reduce the dimension still further. We have $\psi_{1+Z}=l \l$ where $l\in \mN$ and 
$\l$ is a non-trivial linear character of $1+Z$. 
As $(R+Z)^2=0$, the map $x\mapsto 1+x$ is an isomorphism between abelian groups $R+Z$ and $1+R+Z$. Thus, by Lemma~\ref{abchar},
there exists a $1$-dimensional subspace $Y$ of $R+Z$ such that $1+Y\subseteq \ker(\tau\times\l)$, and $Y\ne Z$. 
It follows that
$1+Y\subseteq\ker(\tau\times\psi)$.
Hence a contraction of Type A replaces $(R\oplus D,\tau\times\psi)$ with $((R\oplus D)/Y,\Defl_{1+Y}(\tau\times\psi))$. 
However, the map $x\mapsto x+Y$ is clearly an isomorphism between the AC-pairs $(D,\psi)$ and $((R\oplus D)/Y,\Defl_{1+Y}(\tau\times\psi))$. So $(D,\psi)$ is an offspring, and hence the core, of $(J,\chi)$. 
\end{proof}

\medskip

\begin{center} \textbf{Function} \emph{TypeB} \end{center}
\nopagebreak

\ni\emph{Input:} $(\mbf A,z)$ where $\mbf A=(Q,E,B,R)$ is algebraic data satisfying~\eqref{NZ} and $z\in B$. 
\nopagebreak

\ni\emph{Output:} a categorisation of $\Irr(\mbf A,z)$.

\smallskip
\ni\textbf{Step 1: look for a direct sum decomposition.}
Check whether there exist $x,v\in B$ such that $R(x,v,z)\ne 0$. 
If there are no such $x$ and $v$ then:
\begin{enumerate}
\item Form $\mbf A'$ from $\mbf A$ by removing $z$ from $B$ (and restricting $R$ accordingly).
\item Run \emph{General} on $\mbf A'$, denoting the output by $O'$.
\item Multiply all the character counts in $O'$ by $(q-1)$ and \textbf{return} the resulting categorisation $\mbf O$. 
\end{enumerate}
In the case referred to, there is an obvious one-to-one correspondence $J\mapsto J'$ from $\Cal J(\mbf A,q)$ onto $\Cal J(\mbf A',q)$ such that $J$ is a direct sum of ideals 
$\lan z\ran$ and $J'$. In this situation, by Lemma~\ref{nodiff}, if  $\chi\in\Irr(1+J)$ restricts to $\chi'\in\Irr(1+J')$ then $(J,\chi)$ and $(J',\chi')$ have cores that are either isomorphic or are both small. Thus, assuming $O'$ is a correct categorisation for $\Irr(\mbf A')$, $O$ is correct for $\Irr(\mbf A,z)$. 

\smallskip
\ni\textbf{Step 2: look for a Type B contraction.}
We search for an element $y\in B$ with the following properties:
\begin{enumerate}[(i)]
\item $R(x,y,w)=0$ for all $x,w\in B$  (that is, $Jy=0$ if $J$ is any algebra encoded by $\mbf A$);
\item for some $x\in B$, $R(x,y,z) \ne 0$;
\item $R(y,x,w)=0$ for all $w\ne z$ and all $x\in B$. (Thus, by~\eqref{NZ}, $yJ=\lan z\ran$ for any algebra $J$ in question.)
\end{enumerate}
If there is no such $y\in B$, go to Step 3. 

Suppose that $y\in B$ satisfying these conditions exists. Then $(\lan y\ran,\lan z\ran)$ is a good pair in any algebra $J$ encoded by $\mbf A$: indeed, as $Jy=0$ we have
$C_J (y) y = y C_J (y)=0$; as $yJ=\lan z\ran$ we have $C_J (y) \ne J$.
We now perform a Type B contraction replacing each algebra $J$ encoded by $\mbf A$ with $C_{J}(y)/\lan y\ran$. More precisely, we shall construct an algebraic family 
$\mbf A'=(Q',E',B',R')$ such that, for each $q$, there is a one-to-one correspondence between $J\in \Cal J(\mbf A,q)$ and $C\in\Cal J(\mbf A',q)$ given by
$C=C_J(y)/\lan y\ran$. Let 
$$
\{x_1,\ldots,x_k\}=\{x\in B\!: \; R(y,x,z)\ne 0 \},
$$
with the ordering of $x_1,\ldots,x_k$ inherited from $B$. 

For each $h\in \Cal V(Q,E,q)$, consider the corresponding algebras $J=J(\mbf A,h)$ and $C=C_J(y)/\lan y\ran$, 
and the elements $x'_1,\ldots,x'_{k-1}\in C$ given by 
$$
x'_i=\left( \prod_{a\in R(y,x_k,z)} h(a) \right) x_i- \left( \prod_{b\in R(y,x_i,z)} h(b) \right) x_k, \quad i=1,\ldots,i-1.
$$
Set $B^{-}=B\setminus \{x_1,\ldots,x_k,y\}$ and $B'=B^{-} \sqcup \{x'_1,\ldots,x'_{k-1}\}$, with the ordering on $B'$ defined by taking with the ordering on $B$, replacing 
 $x_i$ with $x'_i$ for $i=1,\ldots,k-1$ and removing $y$ and $x_k$.
Then $B'$ is a basis of $C$: clearly, $B'\subseteq C$ and $|B'|=\dim J-2=\dim C$; and it follows from~\eqref{NZ} that the set $B'$ is linearly independent. 
It remains to encode the structure constants of $C$ with respect to this basis. We abuse notation by ignoring the distinction between elements of 
$Q$ and their images in $\mF_q$ under a substitution $h:Q\ra \mF_q$. We write 
$c_i$ for $\prod_{a\in R(y,x_i,z)} a$, $i=1,\ldots,k$. 
Also, we use the following shorthand: if $v_1,v_2,v_3\in B$,
\begin{equation}\label{shorthand}
P(v_1,v_2,v_3) = 
\begin{cases}
0 & \text{if } R(v_1,v_2,v_3)=0 \\ 
\prod_{a\in R(v_1,v_2,v_3)} a & \text{otherwise}. \\
\end{cases}\
\end{equation}

 We have
$$
uv= \sum_{w\in B'} d_{uvw} w, \quad u,v \in B' 
$$
where $d_{uvw}$ are given by the following expressions:
$$
d_{uvw}=
\begin{cases}
P(u,v,w) & \text{if } u,v,w\in B^{-}, \\
c_k P(u,x_i,w) - c_i P(u,x_k,w) & \text{if } u,w\in B^{-}, v=x'_i, \\
c_k P(x_i,v,w) - c_i P(x_k,v,w) & \text{if } v,w\in B^{-}, u=x'_i,\\
c_k^2 P(x_i,x_j,w) - c_i c_k P(x_k,x_j,w) & \\
- c_j c_k P(x_i,x_k,w) +c_i c_j P(x_k,x_k,w) & \text{if } w\in B^{-}, u=x'_i, v=x'_j,\\
c_k^{-1} P(u,v,x_l) & \text{if } w=x'_l, u,v\in B^{-}, \\
P(u,x_i,x_l) & \text{if } v=x'_i, w=x'_l, u\in B^{-}, \\
P(x_i,v,x_l) & \text{if } u=x'_i, w=x'_l, v\in B^{-}, \\
c_k P(x_i,x_j,x_l) & \text{if } v=x'_i, v=x'_j, w=x'_l. \\
\end{cases}
$$
In all cases $i$, $j$ and $l$ are arbitrary numbers elements of $[1,k-1]$. Note that $P(u,x_k,x_l)=P(x_k,v,x_l)=0$ for all $u\in B$ because
$x_k$ comes after $x_l$ in the ordering of $B$; this makes the formulae in the last three cases simpler than they would otherwise have been.
Now add variables $d_{uvw}$ to
$Q$ and the polynomial equations above to $E$ as necessary. In the fifth case division by $c_k$ has to be encoded, so we add the equation 
\begin{equation}\label{div}
d_{uvx_l}c_k = P(u,v,x_l);
\end{equation}
recall that $c_k$ is the product of variables $a \in R(y,x_k,z)$ and that each $a$ is forced to be nonzero by an inequation from $E$ by~\eqref{NZ}.

In some cases it is not necessary to add new a new variable $d_{uvw}$. This happens if $d_{uvw}=0$ or if $d_{uvw}$ is just a product of variables from $P$ 
(in the latter case, it suffices to set $R'(u,v,w)$ to consist of the variables of which $d_{uvw}$ is the product). Also, if $R(y,x_k,z)\subseteq R(u,v,x_l)$
(with $u,v\in B^{-}$ and $1\le l\le k-1$), an equation~\eqref{div} is not necessary: it suffices to set $R'(u,v,x'_l)=R(u,v,x_l)\sm R(y,x_k,z)$. These simplifications make 
the resulting data $(P',E')$ easier to deal with in some cases. As $d_{uvw}=0$ unless $w$ comes after both $u$ and $v$ in the ordering of $B'$, the tuple $\mbf A'$ that we obtain is indeed algebraic data.

By Corollary~\ref{bijection}, whenever $J$ is an algebra encoded by $\mbf A$ and $C=C_J (y)/\lan y\ran$ is the corresponding algebra encoded by $\mbf A'$, Type B contractions yield a bijection $\Irr(J,\lan z\ran)\ra \Irr(C,\lan z\ran)$. Moreover, if $\chi\in \Irr(J,\lan z\ran)$ is mapped to $\th\in \Irr(C,\lan z\ran)$, then $\chi(1)=q\th(1)$. 
We conclude that $O=((0),O_1)$, with $O_1=(\Irr(\mbf A',z),0,0,1)$, is a correct categorisation for $\Irr(\mbf A,z)$. 

All that remains to do before we apply \emph{TypeB} to $\mbf A'$ recursively is make sure that the data satisfy~\eqref{NZ}.
Run~\emph{SplitIntoCases} on $\mbf A'$, and let $\mbf A'_1,\ldots,\mbf A'_r$ be the output. For each $i=1,\ldots,r$, run \emph{TypeB} on 
$(\mbf A'_i,z)$, $i=1,\ldots,r$, denoting the output by $O'_i$. Construct the aggregate $O'$ of $O'_1,\ldots,O'_r$. Multiply all the character degrees in $O'$ by $t$ and \textbf{return} the resulting categorisation, which is correct for $\Irr(\mbf A,z)$ by Lemma~\ref{aggr}.

\smallskip
\ni\textbf{Step 3: search for a Type A contraction.} 
Let 
$$
M= \{ v\in B\!: \; R(x,v,w)=R(v,x,w)=0 \text{ for all } x,w\in B \}.
$$
Thus, for each $q$ and each $J\in \Cal J(\mbf A,q)$, we have $M=\{v\in B\!:\,  Jv=vJ=0\}$.

Try to find $y\in B$ such that
\begin{enumerate}[(i)]
\item $R(x,y,w)=0$ for all $x,w\in B$; and
\item $R(y,x,w)=0$ for all $w\in B\sm M$ and $x\in B$.
\end{enumerate}
If no such $y$ exists, move to Step 4. Otherwise, set
$$
L= \{ w\in M \! : \; R(y,x,w) \ne 0 \text{ for some } x\in B \}.
$$
There may be more than one possible $y$: we choose $y$ so that $z\in L$ if possible and, subject to this, $|L|$ is as small as it can be. 
Write $\{w_1,\ldots,w_k\}= L\sm \{z\}$ and $U=\lan z,w_1,\ldots,w_k \ran$. 
Consider an algebra $J$ encoded by $\mbf A$ and a character $\chi\in \Irr(1+J,\lan z \ran)$. As $U^2=0$ it follows from Lemma~\ref{abchar} that 
there is a unique subspace $U'\le U$ of codimension $1$ such that
$1+U'\subseteq \ker\chi$. Since $z\notin U'$, there is a uniquely determined tuple $\mbf b=(b_1,\ldots,b_k)\in \mF_q^k$ such that
$$
U'= S(\mbf b) \! = \lan w_1-b_1 z, \ldots, w_k- b_k z \ran.
$$
Now allow $\chi$ to vary and write $H(\mbf b)= \{ \chi\in \Irr(1+J,\lan z\ran)\!:\, 1+S(\mbf b)\subseteq \ker \chi\}$. Then
$$
\Irr(1+J,\lan z\ran) = \bigsqcup_{\mbf b \in \mF_q^k} H(\mbf b)
$$
and, for each fixed tuple $\mbf b$, Type A contractions yield a one-to-one correspondence from $H(\mbf b)$ onto $\Irr(1+(J/S(\mbf b)), \lan z\ran)$. 

We shall construct algebraic data $\mbf A'$ 
that encode precisely the algebras $J/S(\mbf b)$ as $J$ runs through the algebras encoded by $\mbf A$ and $\mbf b$ runs 
through $\mF_q^k$. Then, by the argument above, $((0),(\Irr(\mbf A',z),0,0,0))$ 
will be a correct categorisation of $\Irr(\mbf A,z)$.

We shall use the basis 
$B'=(B\sm L) \sqcup \{z'\}$ of $J/S(\mbf b)$ where $z'=z$. The ordering on $B\sm L$ is inherited from $B$, and we put $z'$ after all the elements of $B\sm L$ in the ordering
of $B'$.
(As usual, we view each $u\in B\sm L$ as an element of $J/S(\mbf b)$ by identifying $u$ with $u+S(\mbf b)$.)
Using the same conventions as in Step 3, including shorthand~\eqref{shorthand}, we can express the products of basic elements in $J/S(\mbf b)$ as
$$
uv= \sum_{w\in B'} d_{uvw} w, \quad u,v\in B'
$$
where
$$
d_{uvw}=
\begin{cases}
P(u,v,w) & \text{if } w \ne z' \\
P(u,v,z)+\sum_{i=1}^k b_i P(u,v,w_i) & \text{if } w=z'.
\end{cases}
$$
Now form algebraic data $\mbf A'=(Q',E',B',R')$ by letting $B'$ be as above, adding formal parameters $b_1,\ldots,b_k$ to $Q$ to obtain $Q'$ 
and expressing the structure constants using the formulae given (adding new parameters and equations as necessary). 
Run \emph{SplitIntoCases} on $\mbf A'$, writing $\mbf A'_1,\ldots,\mbf A'_r$ for the output. 
Finally, run \emph{TypeB} recursively on each of $(\mbf A'_i,z)$, $i\in [1,r]$, and \textbf{return} the aggregate $O$ of the outputs. (Note that 
$L$ was defined so that there is a chance that Step 2 will work on these runs of \emph{TypeB}.) By Lemma~\ref{aggr}, $O$ is correct for
$\Irr(\mbf A,z)$.

\smallskip
\ni\textbf{Step 4: give up.} 
If this stage is reached, return the categorisation $O=((0),O_1)$ where $O_1=(\Irr(\mbf A,z),0,0,0)$.

\medskip
It is easy to construct algebraic data $\mbf A$ encoding the nilpotent triangular algebras $T_n (q)$ for all $q$. 
Then we could run \emph{General} on $\mbf A$. However, in certain cases such as that of $T_n (q)$ we can make the algorithm much more efficient by trying 
a special method of finding offspring and reverting to \emph{General} when that method fails. We describe this special approach, which is also useful for the 
purposes of illustration (cf.\ Section~\ref{examples}), in the next section.

\section{Pattern algebras}\label{spec_cases}

We now show how a number of contractions can be performed at once if we make an assumption concerning the structure of our algebra $J$.

\begin{hyp}\label{hyp31} Let $J$ be a finite-dimensional nilpotent algebra and suppose that as a vector spaces $J=K\oplus L$ where
\begin{enumerate}[(i)]
\item $K$ is an ideal of $J$ and $JK=0$;
\item $L$ is a subalgebra of $J$.
\end{enumerate}
\end{hyp}

\begin{remark} By reversing the order of multiplication, one can also apply the subsequent argument to the case when the condition $JK=0$ is replaced with $KJ=0$.
\end{remark}

As $1+K$ can be identified with the additive group $K$, the set $\Irr(1+K)$ is in one-to-one correspondence with $K^*$ by Lemma~\ref{abchar}.
The ideal $K$ is a right $L$-module, so we may view $K^*$ as a left $L$-module. 
Let $\phi$ be a linear character of $1+K$, and let $u\in K^*$ be the corresponding function. 
Let $M_{\phi}=M_u=\Ann_L (u)$ be the annihilator of $u$ in $L$. Then $\Stab_{1+L}(\phi)=1+M_u$. Indeed, as $LK=0$, 
an element $1+y\in 1+L$ fixes $\phi$ if and only if $\phi(1+x(1+y))=\phi(1+x)$ for all $x\in K$; this occurs if and only if $u(xy)=0$ for all $x\in K$, 
which is equivalent to $yu=0$.

In particular, $\Stab_{1+L}(\phi)$ is an algebra group.
 The linear character 
$\phi$ extends to a unique character $\hat\phi\in (1+K+M_{\phi})$ such that $1+M_{\phi}\subseteq \ker \hat\phi$. 
By Clifford theory, we have a bijection 
\begin{equation}\label{Fphi}
F_{\phi}: \Irr(1+M_{\phi}) \ra \Irr(1+J|\phi), \qquad \th\mapsto (\hat\phi \th)^{1+J}
\end{equation}
 (see, for example,~\cite{CRI}, Propositions 11.5 and 11.8). 

\begin{lem}\label{fast} Assume Hypothesis~\ref{hyp31}. Let $\phi\in\Irr(1+K)$ be non-trivial and let $M=M_{\phi}$. Suppose $\chi\in\Irr(1+J|\phi)$, and 
$\th=F_\phi^{-1}(\chi)$, so that 
$\th\in\Irr(1+M)$. Then $(J,\chi)$ has an offspring of the form $(R\oplus M,\tau\times\th)$ where $R$ is 1-dimensional, $RM=M R=R^2=0$ and $\tau$ is an irreducible character of $R$.
\end{lem}

\begin{proof} We argue by induction on $\dim K$. If $\dim K=1$ then $KL=0$ (as $J$ is nilpotent) and there is nothing to prove. Since $J$ is nilpotent, we can find a $1$-dimensional subspace $Z$ of $K$ such that $ZJ=0$. First, suppose that $1+Z\subseteq \ker\phi$, so $1+Z\subseteq \ker\chi$. Then a Type A contraction sends $(K+L,\chi)$ to 
$((K/Z)\oplus L, \tilde{\chi})$. Hence we obtain the result by applying the inductive hypothesis to the AC-pair $((K/Z)\oplus L,\tilde{\chi})$ 
and the character $\tilde{\phi}\in \Irr(1+K/Z)$ (note that $M=M_{\tilde{\phi}}$).

Thus we may assume that $1+Z\nsubseteq \ker\phi$. Let $Y$ be a $1$-dimensional subspace of $K$ such that $Y\ne Z$ and $YL\subseteq Z$. If in fact $YZ=0$ then we can apply the previous argument with $Z$ replaced by $Y$, so it remains to consider the case $YL=Z$. It follows from Lemma~\ref{abchar}, there is a unique $1$-dimensional subspace
$W$ of $Y+Z$ such that $1+W\subseteq \ker\phi$.
We may assume that $Y$ is this subspace: this can be achieved by replacing $Y$ with 
$\lan y+\a z \ran$, if necessary, where $y$ and $z$ are elements spanning $Y$ and $Z$ respectively and $\a\in \mF_q$ is the appropriate scalar. Observe that the action of the group $1+M=\Stab_{1+L}(\phi)$ preserves $\ker\phi$, and hence $Y$.
Thus, since $MY=0$, we have $YM=0$, whence $M\subseteq C_L (Y)$. The pair $(Z,Y)$ is good and the corresponding Type B contraction replaces $(J,\chi)$ with an AC-pair 
$(J', \om)$ where $J'=(K/Y)\oplus C_L(Y)$ (the direct sum is that of vector spaces).
Since $(\hat\phi\th)^{1+J}=\chi$ and the character $(\hat\phi\th)^{1+K\oplus C_L(Y)}$ is irreducible and contains $1+Y$ in its kernel (as $\hat\phi \th$ does), 
by uniqueness in Lemma~\ref{adjustment}, we have 
$$
(\hat\phi\th)^{1+K\oplus C_L(Y)}=\Infl_{1+Y}(\om).
$$ 
Therefore,
$F_{\tilde{\phi}} (\om)=\th$ (where $\tilde{\phi}$ is the deflation of $\phi$ to $1+K/Y$), 
and we obtain the result by applying the inductive hypothesis to $(J',\om)$ and $\tilde{\phi}$.   
\end{proof}

Throughout the rest of this section we shall assume that $J$ is an even more special kind of algebra: a pattern algebra (see~\cite{Isaacs2007}). 
Let $C$ be a finite set and $R\subseteq C\times C$ a strict partial order on $C$. The \emph{pattern algebra} associated to $(C,R)$ 
is the nilpotent algebra 
$$
T_{C,R}(q) = \{ a=(a_{ij})_{i,j\in C} \in M_{C,C}(q) \! :\; a_{ij}=0 \text{ for all } i,j\in C \text{ such that } (i,j)\notin R \}.
$$
The partially ordered set $(C,R)$ defines the assembly $\Irr(T_{C,R})$, where
$$
\Irr(T_{C,R})(q) = \{ (T_{C,R}(q), \chi) \!: \; \chi\in \Irr(1+T_{C,R}(q)) \}.
$$
We note that $T_n (q)$ is, of course a pattern algebra, corresponding to a totally ordered set with $n$ elements.

We shall construct a certain categorisation of $\Irr(T_{C,R})$. 
This is a modified version of the procedure described by Isaacs~\cite{Isaacs2007}, Section 3, and our exposition will follow a similar pattern. (In order 
to relate~\cite{Isaacs2007} to the present paper, one should reverse all partial orders and replace multiplication in every algebra with the opposite one.)
The only substantial change is that we use
the actions of certain automorphisms of $T_{C,R}(q)$ to reduce the number of cases and make the algorithm more efficient. (Without this improvement, 
the computation for $U_{13}(q)$ would likely take an unreasonable amount of time.)
We shall then combine this categorisation 
with the algorithm of Section~\ref{alg} to find a more ``deep'' categorisation of $\Irr(T_{C,R})$.

First, we need a few general results concerning the algebra $T_{C,R}(q)$.
If $E$ is a subset of $C$, we shall denote by $\Top(E,R)$ the set of maximal elements of $E$ with respect to $R$, and we shall write 
$$
\bar E^R = \{ c\in C \! : \; (c,d) \in R \text{ for some } d\in E \text{ or } c\in E \}.
$$
This is the downward closure of $E$.

Consider the space $\mC^{C}$ of column vectors with entries indexed by $C$, and write $e_c$ for the element of $\mC^C$ whose only non-zero coordinate 
is at $c\in C$ and is equal to $1$.
We will view $\mC^C$ as a left $T_{C,R}(q)$-module, with the usual matrix multiplication.
If $u=\sum_{c\in C}\l_c e_c\in \mC^C$, we define
$$
\Supp(u):= \{ c\in C \!:\; \l_c\ne 0\}
$$ 
to be the \emph{support} of $u$ (as in~\cite{Isaacs2007}).

The following result is effectively a (partial) restatement of~\cite{Isaacs2007}, Theorem 3.1, so we shall not give a proof (which is relatively straightforward).

\begin{lem}\label{orb} Let $R$ be a strict partial order on a finite set $C$. Suppose $u\in \mC^{C}$. Let $E=\Top(\Supp(u),R)$, and set
$$
X= \{ v\in \mC^{C} \!:\; \Supp(v) \subseteq \bar E^{R} \sm E \}.
$$
Then the orbit of $u$ under $1+T_{C,R}(q)$ is $u+X$. In particular, $|(1+T_{C,R}(q))u|=q^{|\bar E^R \sm E|}$.
\end{lem}

Now suppose a total order $<$ on $C$ is given and $<$ contains $R$. 
We define the \emph{normal closure} of $R$ in $<$ to be the relation
$$
\bar{R} = \{ (k,l) \in C\times C \! : \; k<l \text{ and } R\cup {(k,l)} \text{ is a transitive relation} \}.
$$
It is easy to check that $\bar R$ is transitive. 

\begin{lem}\label{normal_cl} Let $R$ be a strict partial order on a finite set $C$. Then for every prime power $q$ the algebra group $1+T_{C,\bar R} (q)$ normalises 
$T_{C,R}(q)$; that is, for all $g\in 1+T_{C,\bar R} (q)$ we have $g^{-1} T_{C,R}(q) g=T_{C,R}(q)$.
\end{lem}

\begin{proof}
For any $i,j\in C$ and $\a \in \mF_q$, let $x_{ij}(\a)=1+\a e_{ij}$, where $e_{ij}$ is the elementary matrix with the $(i,j)$ entry equal to $1$ and the other entries equal to $0$. It is easy to check that, for any partial order $S$ on $C$, the group $1+T_{C,S}(q)$ is generated by the matrices
$x_{ij}(\a)$, $(i,j)\in S$, $\a\in \mF_q$. Thus it suffices to prove that $g^{-1} T_{C,R}(q) g = T_{C,R}(q)$ whenever $g=x_{kl}(\a)$ for some $(k,l)\in \bar R$, $\a\in \mF_q$. Consider such an element $g=x_{kl}(\a)$; note that $g^{-1}=x_{kl}(-\a)$. Let $a=(a_{ij})_{i,j\in C}\in T_{C,R}(q)$, and suppose $m,n \in C$ and $(m,n)\notin R$. Then, by definition of $\bar R$, 
either $n\ne l$ or $(m,k)\notin R$, whence the $(m,n)$ entry of $ag$ is $0$. Hence $ag\in T_{C,R}(q)$.
Moreover, either $m\ne k$ or $(l,n)\notin R$, so the $(m,n)$ entry of $g^{-1} a g$ is also $0$. Therefore, $g^{-1}ag\in T_{C,R}(q)$, and the result follows.
\end{proof}

 Choose a minimal element $c_0$ of $C$ and set 
\begin{eqnarray}
K & = & \{ a\in T_{C,R}(q) \!: \; a_{ij} =0 \text{ for all } i,j\in C \text{ such that } i\ne c_0 \}, \label{eqK} \\
L & = & \{ a\in T_{C,R}(q) \! : \; a_{c_0 j} =0 \text{ for all } j\in C \}. \label{eqL}
\end{eqnarray}
It is easy to check that $K$ and $L$ satisfy Hypothesis~\ref{hyp31}. Our aim is to find a set of representatives of $(1+L)$-orbits $\Irr(1+K)$, and to describe 
the algebra $M_{\phi}$ for each representative $\phi$: then we can find offsprings of characters of $1+T_{C,R}(q)$ using Lemma~\ref{fast} and the discussion before it.

Let $B=C\sm \{c_0\}$ and $P=R\cap (B\times B)$. In the sequel, we shall identify $L$ and $T_{B,P}(q)$ via the obvious isomorphism. 

Consider the space $\mC^D$ of column vectors and write $\{e_d: d\in D\}$ for the standard basis of $\mC^D$.
The projection $T_{B,P}(q) \thra T_{D,R'}(q)$ is an algebra homomorphism (as $a_{dc}=0$ for all $a\in T_{C,R}(q)$, $d\in D$, $c\in C\sm D$), which allows
us to view $\mC^D$ as a left $T_{B,P}(q)$-module.
Moreover, it is easy to see that 
sending $\{e_d:d\in D\}$ to the dual basis of $\{e_{c_0 d}: d\in D \}$ (see the notation in the proof of Lemma~\ref{normal_cl}) yields an isomorphism 
between the left $T_{B,P}(q)$-modules $\mC^D$ and $K^*$. Thus we may identify $\mC^D$ with $\Irr(1+K)$ by composing this isomorphism with that given by Lemma~\ref{abchar}.
In keeping with our previous notation, we shall write 
$$
M_u = \Ann_L (u) = \Ann_{T_{B,P}(q)} (u)
$$
for $u\in \mC^D$.

Let $E$ be any subset of $D$. Set $u_E=\sum_{d\in E} e_d$ and 
$$
M_E (q) = M_{C,R,c_0,E} (q) = M_{u_E}=\Ann_L (u_E).
$$
If $\phi\in \Irr(1+K)$  is the character corresponding to $u_E$, then $\Stab_{1+L}(\phi)=1+M_E(q)$ by the argument preceding Lemma~\ref{fast}. We define the assembly
$\Irr(M_E)=\Irr(M_{C,R,c_0,E})$ as follows:
$$
\Irr(M_{C,R,c_0,E})(q)  =  \{ ( M_E(q), \chi) \! : \; \chi \in \Irr(1+M_E(q)) \}. 
$$

\begin{prop}\label{pr4} Let $R$ be a strict partial order on a finite set $C$. Let $c_0$ be a minimal element of $C$. 
Set $D=\{d\in C \!: \, (c,d)\in R\}$ and $R'=R\cap (D\times D)$. Choose a total order $<$ on $B=C\sm \{c_0\}$ containing $P=R\cap (B\times B)$. Let 
$\bar P$ be the normal closure of $P$ with respect to $<$ and write $\bar P'=\bar P \cap (D\times D)$.
Suppose that $E_1,\ldots,E_n$ is an enumeration of all the anti-chains in $D$ with respect to $\bar P'$. For $i=1,\ldots, n$ set 
$$
O_i = (\Irr(M_{C,R,c_0,E_i}), |E_i|, |\bar E_i^{\bar P'}\sm \bar E_i^{R'}|, |\bar E_i^{R'}\sm E_i|).
$$
Then $((0), O_1,\ldots,O_n)$ is a correct categorisation for $\Irr(T_{C,R})$.
\end{prop}

First, we need the following lemma.

\begin{lem}\label{lempr4} With the notation of Proposition~\ref{pr4}, let $u\in \mC^D$ and $E=\Top(\Supp(u),\bar P')$. 
Then $M_u \cong M_E(q)$ and
$$
|L:M_u| = q^{|\bar E^{R'}\sm E|}.
$$
\end{lem}

\begin{proof} By Lemma~\ref{orb}, the orbit of $u$ under the left action of $1+T_{B,\bar P}(q)$ is $u+X$, where 
$$
X= \{ u\in U \!:\; \Supp(u) \subseteq \bar E^{\bar R} \sm E \}.
$$ 
Hence there exists $g\in 1+T_{B,\bar P}(q)$ such that $\Supp(gu)=E$. By Lemma~\ref{normal_cl}, $g T_{B,P}(q) g^{-1}= T_{B,P}(q)$, i.e.\ $gLg^{-1}=L$. Thus
$$
M_{gu} = \Ann_L (gu)= g M_u g^{-1} \cong M_u.
$$ 
Now write $gu=\sum_{d\in E} \l_d e_d$.
Let $s=(s_{ij})_{i,j\in C\sm \{c_0\}}$ be the diagonal matrix given by
$$
s_{ij} =
\begin{cases}
\l_j^{-1} & \text{if } i=j\in E, \\
1 & \text{if } i=j\notin E, \\
0 & \text{if } i\ne j.
\end{cases}
$$
Then $sLs^{-1}=L$ and $sgu=u_E$, so 
$$
M_{gu}\cong \Ann_L (sgu) =M_E (q).
$$
Thus $M_u\cong M_{gu}\cong M_E (q)$. Now observe that $|L:M_u|=|L:M_E|$ is equal to the size of the orbit of $u_E$ under the action of 
$1+T_{B,P}(q)$ on $\mC^D$. Hence the second statement of the Lemma follows from Lemma~\ref{orb} (applied to the poset $(D,R')$ this time).
\end{proof}

\begin{proof}[Proof of Proposition~\ref{pr4}] Let $K$ and $L$ be as given by~\eqref{eqK} and~\eqref{eqL}.
Let $\phi\in\Irr(1+K)$, and denote by $u$ the corresponding element of $\mC^D$. Define $E(\phi)=\Top(\Supp(u),\bar P')$.  For $i=1,\ldots,n$ let
$$
S_i (q) = \{ (T_{C,R}(q), \chi) \!:\; \chi\in\Irr(T_{C,R}(q)|\phi) \text{ for some } \phi\in \Irr(1+K) \text{ such that } E(\phi)=E_i \}.
$$
By~Lemma~\ref{orb}, 
$E(\phi)=E(\psi)$ whenever $\phi,\psi\in\Irr(1+K)$ are in the same orbit under the action of $1+L$. Hence 
$$
S(q) = S_1 (q) \sqcup \cdots \sqcup S_n (q).
$$
By Lemma~\ref{lempr4}, for each $\phi\in \Irr(1+K)$, we can choose an isomorphism $\sigma_{\phi}:M_{\phi} \ra M_{E(\phi)} (q)$. It induces a bijection
$\sigma_{\phi}:\Irr(1+M_{\phi}) \ra \Irr(1+M_{E(\phi)}(q) )$.

Fix $i\in\{1,\ldots,n\}$ and define a map 
$\Phi_i (q):S_i (q) \ra \Irr(1+M_{E_i} (q))$ as follows. Let $\Omega_i (q)$ be the subset of $\Irr(1+K)$ corresponding, via our identification, to 
$$
\{ u\in \mC^D \! : \; E_i \subseteq \Supp(u) \subseteq E_i \cup (\bar E_i^{\bar P'} \sm \bar E_i^{R'}) \} \subseteq \mC^D
$$
It follows from Lemma~\ref{orb} that $\Omega_i (q)$ is a set of representatives of $(1+L)$-orbits on the set of characters 
$\phi\in\Irr(1+K)$ such that $E(\phi)=E_i$. For each $(T_{C,R}(q),\chi)\in S_i (q)$ there exists a unique $\phi\in \Omega_i(q)$ such that $\lan \chi_{1+K},\phi\ran >0$; define
$$
\Phi_i(q) (T_{C,R}(q),\chi) = (M_{E_i}(q), \sigma_{\phi} (F_{\phi}^{-1}(\chi)) ),
$$
where $F_{\phi}$ is as given by~\eqref{Fphi}.

We claim that the assemblies $S_1,\ldots,S_n$ and the families $\Phi_1,\ldots,\Phi_n$ of maps 
witness the correctness of the categorisation $((0),O_1,\ldots,O_n)$ for $\Irr(T_{C,R})$. 
For $i\in [1,n]$, 
we have $|\Omega_i (q)|=(q-1)^{|E_i|} q^{l_i}$ where
$$
l_i = |\bar E_i^{\bar P'}\sm \bar E_i^{R'}|, 
$$
whence $\Phi_i (q)$ is a $(q-1)^{|E_i|} q^{l_i}$-to-$1$ map, as required. 

Now suppose $\Phi_i (q)$ maps $(T_{C,R}(q),\chi)$ to $(M_{E_i} (q),\th)$. Then 
$$
\chi(1)= |L:M_{E_i} (q)| \th(1) = q^{|\bar E_i^{R'}\sm E_i|} \th(1)
$$
by Lemma~\ref{lempr4}. Also, by Lemmas~\ref{fast} and~\ref{nodiff}, $(T_{C,R}(q),\chi)$ and $(M_{\phi},F_{\phi}^{-1}(\chi))$ have isomorphic cores
 unless $\th$ is trivial, in which case
the cores in question are both small. On the other hand, the AC-pairs $(M_{\phi},F_{\phi}^{-1}(\chi))$ and $(M_{E_i} (q), \th)$ are isomorphic and hence have isomorphic cores. 
Thus,
the cores of $(T_{C,R}(q),\chi)$ and $(M_{E_i} (q),\th)$ are either isomorphic or are both small. This completes the proof of our claim. 
\end{proof}

The following result gives the structure of $M_E (q)$ whenever $E$ is an antichain in the poset $(D,R')$ and $|E|\le 2$. (The first case, $|E|=1$, is a restatement of
a description in~\cite{Isaacs2007}, Section 5).

\begin{lem}\label{smallac}
With the notation of Proposition~\ref{pr4}, let $E$ be an antichain in $D$ with respect to $R'$. Let $B=C\sm \{c_0\}$. If $E=\{k\}$ is a singleton, then 
$M_E(q)\cong T_{B,S}(q)$ for all $q$ where 
$$
S=(R\cap (B\times B)) \sm \{ (d,k) \!:\; d\in D\}.
$$
If $S=\{k,l\}$ is a two-element set, then $M_E (q)$ is isomorphic to the subalgebra of $L$ which has a basis consisting of the following elements:
\begin{eqnarray*}
e'_{ij} & \! := & e_{ij}, \qquad \quad\;  i,j\in B, \;(i,j)\in R \text{ and either } i\notin D \text{ or } j\notin \{k,l \}, \\
f'_i &\! := & e_{ik}+e_{il}, \quad  i\in D, \; (i,k)\in R,\; (i,l)\in R. \\
\end{eqnarray*}
The structure constants with respect to this basis are given by the following identities:
\begin{eqnarray*}
e'_{ij} e'_{rm} & = & \d_{jr} e'_{rm}, \\
e'_{ij} f'_m & = & \d_{jm} f'_i \qquad \qquad \quad  \,\, \text{if } i\in D,\\
e'_{ij} f'_m & = & \d_{jm} (e'_{ik} +e'_{il}) \qquad \text{if } i\notin D, \\
f'_m e'_{ij} & = & (\d_{ik}+\d_{il}) e'_{mj}, \\
f'_i f'_j & = & 0. \\
\end{eqnarray*}
\end{lem}

\begin{proof} If $E=\{k\}$ then
$$
M_E=\Ann_L (e_k) = \{a\in L \! : \; a_{ik}=0 \text{ for all } i\in B\},
$$
so $M_E (q) =T_{B,S}(q)$. Now suppose $R=\{k,l\}$, and let $u=e_k-e_l\in \mC^D$. By Lemma~\ref{lempr4}, $M_E(q)\cong M_u$. 
For each $i\in D$ and $a\in L$, the $i$ entry of $a(e_k-e_l)$ in $e_i$ is $a_{ik}-a_{il}$, so
$$
M_u = \{a\in L \! : \; a_{ik}=a_{il} \text{ for all } i\in D \}.
$$
It is now easy to see that the set given in the statement of the lemma is a basis of $M_u$ and that the given relations between the basis vectors are satisfied.
\end{proof}

We now detail a general algorithm for finding a categorisation of $\Irr(T_{C,R})$, 
which combines the method given by Proposition~\ref{pr4} with the process described in Section~\ref{alg}. If there is an antichain $E$ in the poset $(D,\bar P')$ 
such that $|E|\ge 3$, with the notation as above, then the algorithm just refers the problem to the program described in Section~\ref{alg}. Otherwise, it uses the 
categorisation provided by Proposition~\ref{pr4}, with the algebras appearing in that categorisation expressed as specified in Lemma~\ref{smallac}.

\medskip
\begin{center} \textbf{Function} \emph{PatternAlgebra} \end{center}
\nopagebreak

\ni\emph{Input:} a finite set $C$ and a strict partial order $R$ on $C$.
\nopagebreak

\ni\emph{Output:} a categorisation of $\Irr(T_{C,R})$.

\begin{enumerate}
\item If $R=\varnothing$, \textbf{return} the output $(O_0)$ with $O_0=(q^{|C|})$. Else continue:
\item Choose a minimal element $c_0$ of $C$ (with respect to $R$).
\item Set $D=\{d\in C \!:\, (c_0,d)\in R\}$, $B=C\sm \{c_0\}$ and $P=R\cap (B\times B)$.  
\item Choose a total order $<$ on $B$ containing $P$, and let $\bar P$ be the normal closure of $R$ in $<$. 
(In practice, we choose $<$ so that $|\bar P|$ is as large as possible.)
Set $R'=R\cap (D\times D)$ and $\bar P'=\bar P\cap (D\times D)$.
\item Enumerate all the antichains of the poset $(D,\bar P')$; denote these by $E_1,\ldots,E_n$.
\item If $|E_i|\ge 3$ for some $i$ then construct algebraic data $\mbf A$ such that $\Irr(\mbf A)\cong \Irr(T_{C,R})$ (this is easily done), 
run~\emph{General} on $\mbf A$ and \textbf{return} the output. 
Else continue:
\item\label{runs} For each $i=1,\ldots,n$ proceed as follows.

If $|E_i|=1$, say $E_i=\{d_0\}$, then:

\begin{tabular}{p{1cm}p{11.5cm}}
  & 
Set 
$
S=(R\cap (B\times B)) \sm \{ (d,d_0) \! : \, d\in D\}.
$
Run \emph{PatternAlgebra} recursively on $(B,S)$ and denote the output by $O_i$.
\end{tabular}

If $|E_i|=2$ then:

\begin{tabular}{p{1cm}p{11.5cm}}
 &
Using the basis and relations given in Lemma~\ref{smallac}, construct algebraic data $\mbf A$ such that $\Irr(M_{E_i})\cong \Irr(\mbf A)$ (no parameters are used in $\mbf A$). 
The ordering of the basis set of $\mbf A$ is as follows (in the notation of Lemma~\ref{smallac}): 
$e'_{ij}$ comes before $e'_{rm}$ if and only if $i>r$ or $i=r$ and $j<m$; and $f'_i$ comes before $e'_{jm}$ if and only if $i>j$ or $i=j$ and $\min(k,l)<m$. 
It is straightforward to check that this ordering satisfies the relevant condition of the definition of algebraic data.
Run~\emph{General} on $\mbf A$, and denote the output by $O_i$.
\end{tabular}

\item For each $i=1,\ldots,n$, multiply all the character counts in $O_i$ by $(q-1)^{|E_i|}q^{|\bar E_i^{\bar P'}\sm \bar E_i^{R'}|}$ and all the character degrees in $O_i$ by
$t^{|\bar E_i^{R'}\sm E_i|}$, denoting the result $O'_i$.
\item \textbf{Return} the aggregate of $O'_1,\ldots,O'_n$. By Proposition~\ref{pr4}, Lemma~\ref{smallac} 
and Lemma~\ref{aggr}, this is indeed a correct categorisation for $\Irr(T_{C,R})$ 
(if we assume, inductively, that all the runs of \emph{PatternAlgebra} in Step~\ref{runs} return correct categorisations of their inputs).
\end{enumerate}

\begin{remark} All contractions of Type B performed by the algorithms of this and the previous sections are of a rather special kind: the good pairs $(Z,Y)$ which they employ 
satisfies $JY=0$, where $J$ is the ambient algebra. 
\end{remark}

\section{Output}\label{output}

We now describe the output when the function~\emph{PatternAlgebra} is run on the set $[1,n]$, $n\le 13$, with the usual total order $<$ acting as $R$. 
If $n\le 12$, the output consists of just one family $O_0$, which in its turn contains only a polynomial $f\in \mZ[q,t]$ encoding the numbers $N_{n,e}(q)$ for all $e$ and $q$.
Hence all AC-pairs of the form $(T_n (q),\chi)$, $n\le 12$, have small cores (as claimed in Theorem~\ref{main}).

If $n=13$, the output is of the form $(O_0,O_1,O_2)$. The data $O_0$ again consists just of a single polynomial from $\mZ[q,t]$. For $i=1,2$, we have
$O_i=(F_i,k_i,l_i,m_i)$ with $F_i=\Irr(\mbf A_i,z)$ where $\mbf A_i=(Q_i,E_i,B,R)$, $B=\{y,z\}$, $R(y,y,z)\ne 0$ and all the other products of basis vectors are zero. 
Hence all the AC-pairs $(K,\phi)\in \Irr(\mbf A_i,y)(q)$ satisfy $K\cong x\mF_q[x]/(x^3)$ and $1+K^2\nsubseteq \ker\phi$. 
Clearly, each such pair $(K,\phi)$ is uncontractible. The restrictions $E_2$ are contradictory, so $\mbf A_2$ does not, in fact, encode any AC-pairs.

Let $S_0,S_1,S_2$ be the assemblies witnessing the correctness of the categorisation $O$ for $\Irr(T_{[1,13],<})$. 
We see that the AC-pairs in $S_1(q)$ all have $2$-dimensional cores described in the statement of Theorem~\ref{main}. We have $m_1=16$, 
so all characters $\chi\in \Irr(T_{13}(q))$ lying in $S_1 (q)$ have degree $q^{16}$; and using data returned by the algorithm, we calculate that there are $q(q-1)^{13}$ such characters.
The characters in $S_0 (q)$ all have small cores (for every $q$). This concludes the proof of Theorem~\ref{main}, and hence of Theorem~\ref{wi}. 

The data returned by the program allows us to find $N_{n,e}(q)$ for $n\le 13$ and all values of $e$ and $q$. We observe that, for all $n\le 13$ and all $e$, 
$N_{n,e}(q)$ can be expressed as a polynomial with integer coefficients in $q$, so Conjecture~\ref{conj_poly} is true in these cases. The values of $N_{n,e}(q)$ for $n\le 9$ were already found by Isaacs~\cite{Isaacs2007}. The methods used in \emph{loc.\ cit.}\ do not give a proof that the 
polynomials found were correct; our present computation confirms that they indeed are. The polynomials $N_{n,e}(q)$ for $10\le n\le 13$ are listed in the Appendix.

The running time of the program for $n=13$ is approximately $20$ minutes on a ``standard'' computer (at the time of writing), and for $10\le n \le 14$ the time increases about tenfold each time $n$ is increased by $1$. The output of the algorithm for $n=14$ contains a considerable 
number of ``exceptional'' cases, and a computation of the $N_{14,e}(q)$, $e\in \mN$, while technically feasible, would require further programming. (However, it appears that $N_{14,e}(q)$ are polynomial in $q$ with integer coefficients.) 

\begin{remark} The polynomials given in the Appendix agree with previously available partial data. In particular, the value obtained for $|\Irr(U_n (q))|$ agree
with the numbers of conjugacy classes of $U_n (q)$ found by J.M.\ Arregi and A.\ Vera-Lopez. Also, I.M.\ Isaacs has checked that the values of $N_{13,e}(2)$, $e\ge 0$, calculated by the present algorithm are the same as those returned by an algorithm due to M.\ Slattery (see~\cite{Slattery1986}).
\end{remark}

\begin{remark} The polynomials $N_{n,e}(t+1)$, $n\le 13$, all have nonnegative coefficients in $t$.
\end{remark}

\section{Examples}\label{examples}

Now we demonstrate where the characters of $U_{13} (q)$ with non-small cores come from. We will construct an appropriate example of repeated contractions, 
which are all performed via 
Lemma~\ref{fast}. Throughout, we use the notation of Section~\ref{spec_cases}. We will start with the nilpotent algebra $J_1=T_{13}(q)$ and will construct, recursively, 
a sequence $J_1, \ldots, J_k$ 
as follows. We will find subalgebras $K_i$ and $L_i$ satisfying Hypothesis~\ref{hyp31}.
 We will choose a certain character $\phi_i\in \Irr(1+K)$ and set $J_{i+1}=M_{\phi}$. 
By Lemma~\ref{fast} and Lemma~\ref{nodiff} the composition 
$$
F_{\phi_1} \circ F_{\phi_2} \circ \cdots \circ F_{\phi_{k-1}} 
$$
 is then an injective map
$F:\Irr(1+J_k) \ra \Irr(1+J_1)$ with the property that, for each $\th\in\Irr(1+J_k)$, the AC-pairs 
$(J_1,F(\th))$ and $(J_k,\th)$ have cores that are either isomorphic or are both small. 

Whenever $J_i$ is a pattern algebra $T_{C,R}(q)$, we will choose a minimal element $c_0$ of the poset $(C,R)$ and define $K_i$ and $L_i$ by~\eqref{eqK} and~\eqref{eqL}. If 
$D$ is as in the preceding section, we may identify $\Irr(1+K)$ with $\mC^D$, so we shall specify $\phi_i$ by an element $u_i\in \mC^D$.
Posets $(C,R)$ will be represented by diagrams. Each diagram will consist of nodes corresponding to elements of $C$ and edges connecting those nodes. We say that $i$ is less than $j$ 
if $i$ and $j$ are connected by an edge and $i$ is located to the left of $j$; and we take $R$ to be the transitive closure of this relation.
Thus our original algebra $J_1=T_{13}(q)$ corresponds to the diagram shown in Fig.~\ref{step0}.

\begin{figure}\caption{The initial algebra $J_1$}\label{step0}
\begin{center}
\includegraphics{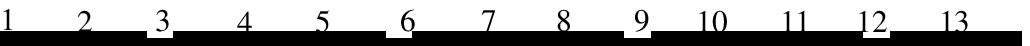}
\end{center}
\end{figure}

\ni\textbf{Step 1.} We set $c_0=1$, so $D=[2,13]$. Choose $u_1=\l_1 e_5$ where $\l_1\in\mF_q^{\times}$ is an arbitrary non-zero scalar. By Lemma~\ref{smallac} (and Lemma~\ref{lempr4}), the 
resulting algebra $J_2 = M_{u_1}$ is the pattern algebra given by Fig.~\ref{step1}.

\begin{figure}\caption{The algebra $J_2$}\label{step1}
\begin{center}
\includegraphics{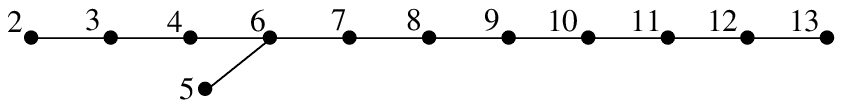}
\end{center}
\end{figure}

\ni\textbf{Step 2.} Now let $c_0=2$ and choose $u_2=\l_2 e_6$, $\l_2\in \mF_q^{\times}$. Then $J_3=M_{u_2}$ is given by the diagram shown in Fig.~\ref{step2}.

\begin{figure}\caption{The algebra $J_3$}\label{step2}
\begin{center}
\includegraphics{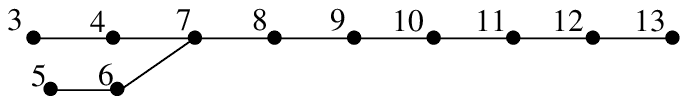}
\end{center}
\end{figure}

\ni\textbf{Step 3.} We set $c_0=3$ and pick $u_3=\l_3 e_{10}$, $\l_3\in \mF_q^{\times}$. The resulting algebra $J_4=M_{u_3}$ is given by Fig.~\ref{step3}.

\begin{figure}\caption{The algebra $J_4$}\label{step3}
\begin{center}
\includegraphics{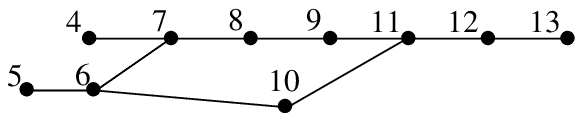}
\end{center}
\end{figure}

\ni\textbf{Step 4.} Let $c_0=4$ and $u_4=\l_4 e_{11}$, $\l_4\in \mF_q^{\times}$. We obtain the pattern algebra $J_5 = M_{u_4}$, which corresponds to the poset shown 
in Fig.~\ref{step4}.

\begin{figure}\caption{The algebra $J_5$}\label{step4}
\begin{center}
\includegraphics{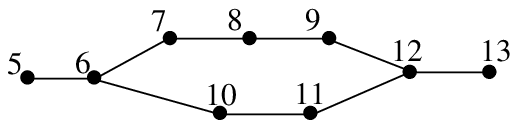}
\end{center}
\end{figure}

From now on, we shall use additional notation. Suppose that a poset $(C,R)$ is given by a diagram as before. If $2m$ distinct edges $(i_s,j_s)$, $(k_s,l_s)$, $s=1,\ldots,m$,
are labelled so that the labels on $(i_1,j_1),\ldots,(i_m,j_m)$ are all distinct and the labels on $(i_s,j_s)$ and $(k_s,l_s)$ are the same for each $s\in [1,m]$ then the 
vector subspace of $T_{C,R}(q)$ represented by the diagram is
$$
\{ a=(a_{ij})_{i,j\in C} \in T_{C,R}(q) \! : \; a_{i_s j_s} = a_{k_s l_s} \text{ for all } s\in [1,m] \}.
$$
In all our examples this subspace will actually be a subalgebra.

\ni\textbf{Step 5.} Set $c_0=5$ and choose $u_5=\l_5 e_{11}+ \l_6 e_8$, where $\l_5,\l_6\in\mF_q^{\times}$. By the proof of Lemma~\ref{smallac} and by Lemma~\ref{lempr4}, 
 the algebra $J_6=M_{u_5}$ is isomorphic to the one represented by Fig.~\ref{step5}.

\begin{figure}\caption{The algebra $J_6$}\label{step5}
\begin{center}
\includegraphics{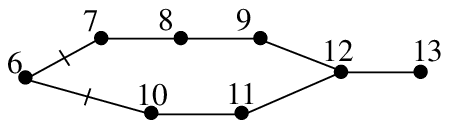}
\end{center}
\end{figure}

(The isomorphism is the identity map if $\l_5=1$ and $\l_6=-1$.)

As we are no longer dealing with a pattern algebra, the following steps will be described in the more general framework of the discussion before Lemma~\ref{fast}. 

\ni\textbf{Step 6.}
As before, we set
\begin{eqnarray*}
K_6 & = & \{ a\in J_6 \!: \; a_{ij} =0 \text{ for all } i,j\in [6,13] \text{ such that } i\ne 6 \},  \\
L_6 & = & \{ a\in J_6 \! : \; a_{6 j} =0 \text{ for all } j\in [7,13] \}. 
\end{eqnarray*}
Let $V=\lan e_{6j} \ran_{j\in [7,13]}$. Then
$$
K_6= \lan e_{6,7}+e_{6,10}, e_{6,8}, e_{6,9}, e_{6,11}, e_{6,12}, e_{6,13} \ran \subseteq V.
$$
By an observation in Section~\ref{spec_cases}, $V^*$ is isomorphic to $\mC^{[7,13]}$ as a left $L_6$-module, whence we have an $L_6$-module isomorphism
$$
K_6^* \cong \mC^{[7,13]}/K_6^{\perp} = \mC^{[7,13]}/ \lan e_7-e_{10} \ran. 
$$
So we have a one-to-one correspondence between $\Irr(1+K_6)$ and $\mC^{[7,13]}/ \lan e_7-e_{10} \ran$. Let $\phi_6\in\Irr(1+K_6)$ be the character corresponding to
$u_6= \l_7 e_8 + \l_8 e_{11} + \lan e_7-e_{10} \ran$, where $\l_7,\l_8\in\mF_q^{\times}$. 

Arguing as in the proof of Lemma~\ref{lempr4}, we consider the diagonal matrix $s=(s_{ij})_{i,j\in [7,13]}$ with $s_{8,8}=\l_7^{-1}$, $s_{11,11}=-\l_8^{-1}$ and $s_{ii}=1$ for $i\ne 8,11$.
Observe that $sL_6 s^{-1} = L_6$, so 
$$
M_{u_6} \cong M_{s u_6}= M_{e_8-e_{11}+\lan e_7+e_{10} \ran}.  
$$
It follows that $J_7\cong \Ann_L (e_8 - e_{11}+ \lan e_7+e_{10} \ran)$, whence $J_7$ is represented by Fig.~\ref{step6}.

\begin{figure}\caption{The algebra $J_7$}\label{step6}.
\begin{center}
\includegraphics{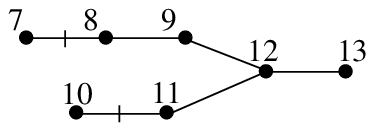}
\end{center}
\end{figure}

\ni\textbf{Step 7.} 
Let $A=\{7,10\}$ and $B=\{8,9,11,12,13\}$, and set
\begin{eqnarray*}
K_7 & = & \{ a\in J_7 \!: \; a_{ij} =0 \text{ for all } i,j\in B \},  \\
L_7 & = & \{ a\in J_7 \! : \; a_{ij} =0 \text{ for all } i\in A \text{ and } j\in B  \},
\end{eqnarray*}
so that Hypothesis~\ref{hyp31} is satisfied. Proceeding similarly to the previous step, 
we can embed $K_7$ into the vector space $V=\lan e_{7,j}, e_{10,j} \ran_{j\in B}=M_{A,B}(q)$. Then $V^*$ may be identified with
$M_{B,A}(q)$ (so that $(u,v)=\trace(uv)$ for $v\in V$, $u\in M_{B,A}(q)$), whence the left $L_7$-module $K_7^*$ is isomorphic to 
$$
M_{B,A}(q)/K_7^{\perp} = M_{B,A}(q)/ \lan e_{8,10}, e_{9,10}, e_{11,7}, e_{11,10}-e_{8,7} \ran.
$$
Let $\phi_7\in\Irr(1+K)$ be the character corresponding to the image of $\l_9 e_{9,8} +\l_{10} e_{12,11}$ in $M_{B,A}(q)/K_7^{\perp}$ (which we identify with $K_7^*$), 
where $\l_9,\l_{10}\in \mF_q^{\times}$. 
Arguing as in Step 6, we infer that the isomorphism class of the algebra $J_8=M_{\phi_8}$ does not depend on the choice of $\l_9$ and $\l_{10}$. 
Setting $\l_9=-\l_{10}=1$, 
we see that $J_8$ is represented by the diagram Fig.~\ref{step7}.

\begin{figure}\caption{The algebra $J_8$}\label{step7}
\begin{center}
\includegraphics{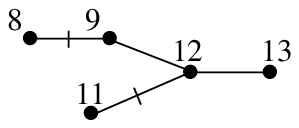}
\end{center}
\end{figure}

\ni\textbf{Step 8.}
We let $A=\{8,11\}$, $B=\{9,12,13\}$ and proceed exactly as in the previous step. 
We have an isomorphism of $L_8$-modules: 
$$
K_8^* \cong M_{B,A}(q)/K_8^{\perp} = M_{B,A}(q) / \lan e_{9,11}, e_{9,8}-e_{12,11} \ran
$$
Let $\phi_8\in\Irr(1+K_8)$ correspond to the image of $\l_{11} e_{12,9}+\l_{12} e_{13,12}$ in $M_{B,A}(q)/K_8^{\perp}$,
where $\l_{11},\l_{12}\in \mF_q^{\times}$. Then $J_9=M_{\phi_8}$ can be represented by Fig.~\ref{step8}.

\begin{figure}\caption{The algebra $J_9$}\label{step8}
\begin{center}
\includegraphics{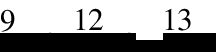}
\end{center} 
\end{figure}

So 
$$
J_9 \cong \left\{ 
\begin{pmatrix}
0 & y & z \\
0 & 0 & y \\
0 & 0 & 0 \\
\end{pmatrix} \! : \; y,z \in \mF_q
\right \},
$$
and hence $J_9\cong x\mF_q[x]/(x^3)$. As we observed in Section~\ref{spec_cases},
 the AC-pairs $(J_9,\chi)$ with $1+J_9^2\nsubseteq \ker\chi$ are uncontractible, so the image of $\Irr(1+J_9,J_9^2)$ under the injective map $F:\Irr(1+J_9)\ra \Irr(1+J_1)$ 
consists of characters with non-small cores. Note that if any of the scalars $\l_1,\ldots,\l_{12}$ is altered then
 for some $i\in [1,8]$ the corresponding character $\phi_i$ changes to a character of $1+K_i$ which is not in the same orbit under the action of $1+L_i$. Hence, 
if $F'$ is obtained by changing $F$ accordingly, then the images of $F$ and $F'$ are disjoint. Since there are $(q-1)^{12}$ choices of scalars $\l_1,\ldots,\l_{12}$ and 
$|\Irr(1+J_8,J_8^2)|= q(q-1)$, the process above accounts for $q(q-1)^{13}$ characters of $U_{13} (q)$. By Theorem~\ref{main}, these are in fact all the characters of
$U_{13}(q)$ with non-small cores. 

\bigskip

We now construct an example proving Theorem~\ref{size25}. 
In terms of diagrams, the idea is to start with the poset $([1,25],<)$, do the 8 steps above on elements $1,\ldots,13$ keeping the rest of the poset unchanged and then
perform the symmetric reflection of the same 8 steps on entries $25,\ldots,13$. 

Our starting point is the example above, which can be reformulated as follows. 
Each algebra $J_i$, $i=1,\ldots,9$, is embedded, in the obvious way, into a matrix algebra, 
which we may write as $\End_{\mC}(R_i)$ where $R_i$ is the corresponding vector space of column vectors (for example, $R_1=\lan e_1,\ldots,e_{13} \ran$ and 
$R_9=\lan e_9, e_{12}, e_{13} \ran$). 
Observe that, for each $i\in [1,8]$, there is a direct sum decomposition $R_i=Q_i \oplus R_{i+1}$ such that
\begin{eqnarray}
K_i & = & \{ a\in J_i \!: \; \im (a) \subseteq Q_i \}, \label{Kinew} \\
L_i & = & \{ a\in J_i \! : \; a(R_{i+1}) \subseteq R_{i+1} \}. \label{Linew}
\end{eqnarray}
For example, if $i\le 5$ then $Q_i=\lan e_{c_0} \ran$ (for the corresponding $c_0$) and $R_{i+1}=\lan e_j \ran_{j\in C\sm \{c_0\}}$.

Let $J'_1=T_{25}(q)$, so that $J'_1\subseteq \End(R_1\oplus W)$, where $W=\lan e_{14},\ldots,e_{25} \ran$. Let $R'_i=R_i\oplus W$ for $i=1,\ldots,9$. 
We perform the same Steps 1--8 as above, 
except that we now start with the algebra $J'_1$ and replace each $R_i$ with $R'_i$. More precisely, suppose we have reached an algebra 
$J'_i\subseteq \End(R'_i)$ after $i-1$ steps. Set $Q'_i=Q_i$ and define $K'_i$ and $L'_i$ by~\eqref{Kinew} and~\eqref{Linew}, with $Q_i$ and $R_{i+1}$ replaced by $Q'_i$ and $R'_{i+1}$ respectively. We make the following inductive assumption:
\begin{equation}\label{Ji}
J'_i \supseteq \{ a\in \End(R'_i) \! : \; \ker(a) \supseteq R_i \text{ and } \im(a) \subseteq W \}.
\end{equation}
(Note that this holds when $i=1$.) Observe that $K'_i$ is a direct sum of $K_i$ and
$$
X_i = \{a\in \End(R_i') \! : \; \ker(a) \supseteq R_i \text{ and } \im(a) \subseteq Q_i\}.  
$$
We define $u'_i\in (K'_i)^*$ by extending $u_i\in K_i^*$ so that $u'_i (X_i)=0$. Then 
$$
J'_{i+1}=\Ann_{L'_i} (u'_i) = \{ a\in \End(R_{i+1} \oplus W) \! : \; a(R_{i+1})\subseteq R_{i+1}, \; a|_{R_{i+1}} \in M_{u_i} \text{ and } a \in \pi(J'_i) \}, 
$$
where $\pi: \End(R_i\oplus W)\ra \End(R_{i+1}\oplus W)$ is the linear map given by removing the rows and columns corresponding to $Q_i$. 
We see that $J'_{i+1}$ satisfies~\eqref{Ji}, and this implies that $K'_{i+1}$ and $L'_{i+1}$ satisfy Hypothesis~\ref{hyp31} (insofar as $K_{i+1}$ and $L_{i+1}$ do).
After 8 steps we obtain the algebra $J'_9$ given by Fig.~\ref{n25step8}

\begin{figure}\caption{The algebra $J'_9$}\label{n25step8}.
\begin{center}
\includegraphics{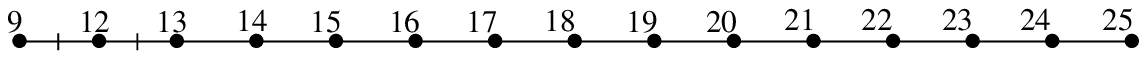}
\end{center}
\end{figure}

We repeat the same 8 steps setting the new $J'_1$ to be this algebra, reversing the order of multiplication (see the remark after Hypothesis~\ref{hyp31}), replacing $e_1,\ldots,e_{13}$ with $e_{25},\ldots,e_{13}$ respectively (so that $R_1=\lan e_{25},\ldots,e_{13} \ran$) and letting $W=\lan e_{9}, e_{12} \ran$. Note that condition~\eqref{Ji} is then satisfied for the $J'_1$, so the process described in the previous paragraph does apply in this situation. The resulting algebra $E$ is given by Fig.~\ref{n25step16}.

\begin{figure}\caption{The algebra $E$}\label{n25step16}.
\begin{center}
\includegraphics{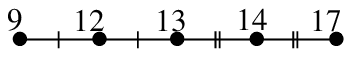}
\end{center}
\end{figure}

As in the first example of this section, we have an injective map $F:\Irr(1+E)\ra \Irr(U_{25}(q))$ such that, for all $\th\in\Irr(1+E)$, 
 the AC-pairs $(E,\th)$ and $(T_{25}(q),F(\th))$ have cores that are either isomorphic or are both small. 

We identify $E$ with the algebra consisting of all matrices $a\in T_5 (q)$ such that $a_{12}=a_{23}$ and $a_{34}=a_{45}$. Let 
$$
I=\left\{ \begin{pmatrix}
0 & 0 & a_{13} & * & * \\
0 & 0 & 0 &  a_{24} & * \\
0 & 0 & 0 & 0 & a_{35} \\
0 & 0 & 0 & 0 & 0 \\
0 & 0 & 0 & 0 & 0 \\
\end{pmatrix} 
\! : \; a_{13}+a_{24}+a_{35}=0
\right\},
$$
where the symbols $*$ denote arbitrary elements of $\mF_q$ that are independent of each other.

We consider the case $q=2$. Then $1+(E/I)$ is isomorphic to the quaternion group of order $8$. So $1+(E/I)$ has a unique irreducible character $\th$ of degree $2$, and $\th$ is real-valued but is not of real type. Hence, by Proposition~\ref{char2eqv}, $F(\th)\in\Irr(U_{25}(q))$ is also real-valued but is not of real type. 
This proves Theorem~\ref{size25}. We remark that if $q\ne 2$ then irreducible characters of $1+(E/I)$ behave differently: they all have small cores.

\begin{appendix}

\section{Appendix: some values of  $N_{n,e}(q)$} 
We list all the non-zero polynomials $N_{n,e}(q)$ for $10\le n\le 13$. The values for $n\le 9$ are given in~\cite{Isaacs2007}.

\doublespacing

\begin{longtable}{|c|l|}
\multicolumn{2}{c}{$n=10$}\\
\hline
$e$ & \multicolumn{1}{c|}{ $N_{n,e}(q)$}  \\
\hline\endfirsthead
\multicolumn{2}{c}{$n=10$ (continued)}\\
\hline
$e$ & \multicolumn{1}{c|}{ $N_{n,e}(q)$}  \\
\hline\endhead
\multicolumn{2}{c}{\textit{Continued on the next page}}
\endfoot
\endlastfoot
0 & $q^9$ \\
\hline
1 & $7q^9 - 6q^8 - q^7$ \\
\hline
2 & $ 5q^{10} + 12q^9 - 26q^8 + 3q^7 + 6q^6$ \\
\hline
3 & $3q^{11} + 12q^{10} + 4q^9 - 47q^8 + 15q^7 + 23q^6 - 10q^5$ \\
\hline
4 & $q^{12} + 6q^{11} + 25q^{10} - 20q^9 - 77q^8 + 59q^7 + 33q^6 - 30q^5 + 3q^4$ \\
\hline
5 & $2q^{12} + 19q^{11} + 21q^{10} - 58q^9 - 87q^8 + 127q^7 + 26q^6 - 63q^5 + 10q^4 + 3q^3$ \\
\hline
6 & $ 6q^{12} + 21q^{11} + 37q^{10} - 137q^9 - 45q^8 + 182q^7 + 6q^6 - 109q^5 + 40q^4 - q^2$ \qquad \qquad \qquad \\
\hline
7 & \!\!\!\!\! 
$\begin{array}{l}
2q^{13} + 3q^{12} + 34q^{11} + 35q^{10} - 204q^9 - 3q^8 + 261q^7 - 54q^6 - 134q^5 \\
+ 62q^4 - q^3 - q^2
\end{array}$ \\
\hline
8 & \!\!\!\!\!
$\begin{array}{l} 
q^{13} + 11q^{12} + 27q^{11} + 28q^{10} - 204q^9 - 35q^8 + 372q^7 - 137q^6 - 143q^5 \\
+ 82q^4 + 4q^3 - 6q^2 \\
\end{array}$ \\
\hline
9 & \!\!\!\!\!
$\begin{array}{l} 
10q^{12} + 45q^{11} - 3q^{10} - 205q^9 - 32q^8 + 424q^7 - 159q^6 - 189q^5 \\
+ 114q^4 + q^3 - 6q^2
\end{array}$   \\
\hline
10 & \!\!\!\!\!
$\begin{array}{l} 
2q^{13} + 4q^{12} + 36q^{11} + 20q^{10} - 203q^9 - 57q^8 + 442q^7 - 138q^6 - 253q^5 \\
+ 159q^4 + q^3 - 13q^2 - q + 1 \\
\end{array}$ \\
\hline
11 & \!\!\!\!\!
$\begin{array}{l} 
6q^{12} + 19q^{11} + 53q^{10} - 179q^9 - 147q^8 + 478q^7 - 85q^6 - 297q^5 \\
+ 161q^4 - 2q^3 - 6q^2 - q \\
\end{array}$ \\
\hline
12 & \!\!\!\!\!
$\begin{array}{l} 
2q^{12} + 18q^{11} + 35q^{10} - 86q^9 - 232q^8 + 432q^7 + 35q^6 - 372q^5 + \\
+ 160q^4 + 29q^3 - 22q^2 + q \\
\end{array}$ \\
\hline
13 & \!\!\!\!\!
$\begin{array}{l} 
10q^{11} + 30q^{10} - 30q^9 - 208q^8 + 226q^7 + 228q^6 - 394q^5 \\
+ 105q^4 + 60q^3 - 29q^2 + 2q \\
\end{array}$ \\
\hline
14 & \!\!\!\!\!
$\begin{array}{l} 
 2q^{11} + 20q^{10} + 7q^9 - 134q^8 + 30q^7 + 286q^6 - 246q^5 - 46q^4 \\
+ 110q^3 - 23q^2 - 9q + 3 \\
\end{array}$ \\
\hline
15 &  
$4q^{10} + 30q^9 - 68q^8 - 67q^7 + 205q^6 - 44q^5 - 145q^4 + 98q^3 - 8q^2 - 5q$ 
 \\
\hline
16 & 
$
10q^9 + 7q^8 - 90q^7 + 86q^6 + 68q^5 - 125q^4 + 29q^3 + 27q^2 - 13q + 1
$ \\
\hline
17 & 
$12q^8 - 22q^7 - 27q^6 + 85q^5 - 47q^4 - 20q^3 + 23q^2 - 3q - 1$ \\
\hline
18 &
$
9q^7 - 28q^6 + 21q^5 + 16q^4 - 28q^3 + 9q^2 + 2q - 1 
$ \\
\hline
19 &
$
4q^6 - 15q^5 + 20q^4 - 10q^3 + q
$ \\
\hline
20 & 
$
q^5 - 4q^4 + 6q^3 - 4q^2 + q
$ \\
\hline
\end{longtable}

\begin{longtable}{|c|l|}
\multicolumn{2}{c}{$n=11$}\\
\hline
$e$ & \multicolumn{1}{c|}{$N_{n,e}(q)$} \\
\hline\endfirsthead
\multicolumn{2}{c}{$n=11$ (continued)}\\
\hline
$e$ & \multicolumn{1}{c|}{$N_{n,e}(q)$} \\
\hline\endhead
\multicolumn{2}{c}{\textit{Continued on the next page}}
\endfoot
\endlastfoot
0 &
$
q^{10}
$ \\
\hline
1 & 
$
8q^{10} - 7q^9 - q^8
$\\
\hline
2 & 
$
6q^{11} + 17q^{10} - 37q^9 + 7q^8 + 7q^7
$\\
\hline
3 & 
$
4q^{12} + 19q^{11} + 3q^{10} - 74q^9 + 35q^8 + 28q^7 - 15q^6
$\\
\hline
4 & 
$
2q^{13} + 11q^{12} + 35q^{11} - 36q^{10} - 124q^9 + 118q^8 + 40q^7 - 55q^6 + 9q^5
$\\
\hline
5 & 
$
8q^{13} + 25q^{12} + 40q^{11} - 126q^{10} - 129q^9 + 263q^8 + q^7 - 116q^6 + 31q^5 + 3q^4
$\\
\hline
6 & 
$ \!\!\!\!\!
\begin{array}{l}
2q^{14} + 10q^{13} + 63q^{12} - 2q^{11} - 246q^{10} - 27q^9 + 396q^8 - 90q^7 - 204q^6 \\
+ 106q^5 - 5q^4 - 3q^3 \\
\end{array}
$\\
\hline
7 & \!\!\!\!\!
$\begin{array}{l}
4q^{14} + 28q^{13} + 70q^{12} - 50q^{11} - 383q^{10} + 190q^9 + 480q^8 - 264q^7 - 237q^6 \\
+ 193q^5 - 29q^4 - 2q^3 \\
\end{array}
$\\
\hline
8 & \!\!\!\!\!
$
\begin{array}{l}
2q^{15} + 6q^{14} + 32q^{13} + 80q^{12} - 63q^{11} - 576q^{10} + 470q^9 + 557q^8 - 523q^7 - 192q^6 \\
+ 259q^5 - 45q^4 - 9q^3 + 2q^2 \\
\end{array}
$\\
\hline
9 & \!\!\!\!\!
$
\begin{array}{l}
q^{15} + 12q^{14} + 43q^{13} + 94q^{12} - 151q^{11} - 672q^{10} + 726q^9 + 671q^8 - 876q^7 - 108q^6 \\
+ 341q^5 - 71q^4 - 11q^3 + q^2 \\
\end{array}
$ \\
\hline
10 & \!\!\!\!\!
$\begin{array}{l}
 2q^{15} + 15q^{14} + 44q^{13} + 110q^{12} - 239q^{11} - 703q^{10} + 964q^9 + 684q^8 - 1198q^7 \\
+ 12q^6 + 453q^5 - 134q^4 - 15q^3 + 4q^2 + q \\
\end{array}$ \\
\hline
11 & \!\!\!\!\!
$\begin{array}{l}
2q^{15} + 10q^{14} + 69q^{13} + 86q^{12} - 290q^{11} - 647q^{10} + 986q^9 + 781q^8 - 1422q^7 \\
+ 84q^6 + 530q^5 - 187q^4 - 8q^3 + 5q^2 + q \\
\end{array}$ \\
\hline
12 & \!\!\!\!\!
$\begin{array}{l}
2q^{15} + 11q^{14} + 50q^{13} + 111q^{12} - 277q^{11} - 686q^{10} + 967q^9 + 936q^8 - 1581q^7 \\
+ 66q^6 + 613q^5 - 198q^4 - 23q^3 + 8q^2 + q
\end{array}$ \\
\hline
13 & \!\!\!\!\!
$\begin{array}{l}
2q^{15} + 8q^{14} + 40q^{13} + 93q^{12} - 165q^{11} - 716q^{10} + 733q^9 \\
 + 1214q^8 - 1579q^7 - 191q^6 + 842q^5 - 261q^4 - 39q^3 + 19q^2 \\
\end{array}$ \\
\hline
14 & \!\!\!\!\!
$\begin{array}{l}
8q^{14} + 24q^{13} + 88q^{12} - 53q^{11} - 736q^{10} + 450q^9 + 1473q^8 - 1467q^7 - 469q^6 \\
+ 968q^5 - 250q^4 - 55q^3 + 17q^2 + 3q - 1 \\
\end{array}$ \\
\hline
15 & \!\!\!\!\!
$\begin{array}{l}
29q^{13} + 57q^{12} + 5q^{11} - 567q^{10} + 72q^9 + 1507q^8 - 1043q^7 - 856q^6 \\
+ 1053q^5 - 173q^4 - 116q^3 + 30q^2 + 2q \\
\end{array}$\\
\hline
16 & \!\!\!\!\!
$\begin{array}{l}
q^{14} + 2q^{13} + 74q^{12} + 11q^{11} - 359q^{10} - 228q^9 + 1298q^8 - 479q^7 - 1054q^6 \\
+ 870q^5 - 32q^4 - 123q^3 + 14q^2 + 5q \\
\end{array}$\\
\hline
17 & \!\!\!\!\!
$\begin{array}{l}
4q^{13} + 15q^{12} + 83q^{11} - 192q^{10} - 390q^9 + 888q^8 + 106q^7 - 1083q^6 + 538q^5 \\
+ 180q^4 - 167q^3 + 9q^2 + 10q - 1 \\
\end{array}$\\
\hline
18 & \!\!\!\!\!
$\begin{array}{l}
9q^{12} + 31q^{11} + 16q^{10} - 382q^9 + 300q^8 + 655q^7 - 984q^6 + 208q^5 \\
+ 277q^4 - 135q^3 + q^2 + 3q + 1 \\
\end{array}$\\
\hline
19 & \!\!\!\!\!
$\begin{array}{l}
16q^{11} + 40q^{10} - 161q^9 - 126q^8 + 669q^7 - 462q^6 - 253q^5 \\
+ 407q^4 - 122q^3 - 21q^2 + 15q - 2 \\
\end{array}$\\
\hline
20 & 
$
26q^{10} - 4q^9 - 216q^8 + 303q^7 + 126q^6 - 521q^5 + 368q^4 - 71q^3 - 16q^2 + 5q
$ \\
\hline
21 & 
$
31q^9 - 70q^8 - 67q^7 + 325q^6 - 299q^5 + 152q^3 - 90q^2 + 19q - 1
$ \\
\hline
22 & 
$
25q^8 - 89q^7 + 79q^6 + 79q^5 - 199q^4 + 141q^3 - 35q^2 - 3q + 2
$ \\
\hline
23 & 
$
 14q^7 - 61q^6 + 97q^5 - 55q^4 - 20q^3 + 41q^2 - 19q + 3
$ \\
\hline
24 & 
$
 5q^6 - 24q^5 + 46q^4 - 44q^3 + 21q^2 - 4q 
$ \\
\hline
25 &
$
q^5 - 5q^4 + 10q^3 - 10q^2 + 5q - 1
$ \\
\hline
\end{longtable}

\begin{longtable}{|c|l|}
\multicolumn{2}{c}{$n=12$} \\
\hline
$e$ & \multicolumn{1}{c|}{$N_{n,e}(q)$} \\
\hline\endfirsthead
\multicolumn{2}{c}{$n=12$ (continued)} \\
\hline
$e$ & \multicolumn{1}{c|}{$N_{n,e}(q)$} \\
\hline\endhead
\multicolumn{2}{c}{\textit{Continued on the next page}}
\endfoot
\endlastfoot
0 & 
$
q^{11}
$\\
\hline
1 &
$
9q^{11} - 8q^{10} - q^9
$ \\
\hline
2 &
$
7q^{12} + 23q^{11} - 50q^{10} + 12q^9 + 8q^8
$ \\
\hline
3 &
$
5q^{13} + 28q^{12} + 3q^{11} - 113q^{10} + 67q^9 + 31q^8 - 21q^7
$ \\
\hline
4 &
$
3q^{14} + 19q^{13} + 48q^{12}  - 62q^{11} - 188q^{10} + 212q^9 + 39q^8 - 90q^7 + 19q^6
$ \\
\hline
5 &
$
q^{15} + 12q^{14} + 44q^{13} + 57q^{12} - 233q^{11} - 180q^{10} + 491q^9 - 72q^8 - 193q^7 + 73q^6
$ \\
\hline
6 & \!\!\!\!\!
$
\begin{array}{l}
4q^{15} + 35q^{14} + 74q^{13} - 7q^{12} - 463q^{11} + 46q^{10} + 767q^9 - 333q^8 - 326q^7 \\
 + 232q^6 - 23q^5 - 6q^4 \\
\end{array}
$
\\
\hline
7 & \!\!\!\!\!
$\begin{array}{l}
18q^{15} + 49q^{14} + 137q^{13} - 224q^{12} - 648q^{11} + 547q^{10} + 899q^9 - 788q^8 - 343q^7 \\
 + 464q^6 - 112q^5 + q^3 \\
\end{array}$
\\
\hline
8 & \!\!\!\!\!
$\begin{array}{l}
q^{17} + 4q^{16} + 14q^{15} + 120q^{14} + 124q^{13} - 519q^{12} - 716q^{11} + 1280q^{10} + 758q^9 \\ 
 - 1390q^8 - 126q^7 + 663q^6 - 218q^5 + q^4 + 4q^3 \\
\end{array}$ \\
\hline
9 & \!\!\!\!\!
$\begin{array}{l}
11q^{16} + 43q^{15} + 116q^{14} + 120q^{13} - 758q^{12} - 904q^{11} + 2374q^{10} + 259q^9 - 2049q^8 \\
 + 327q^7 + 799q^6 - 365q^5 + 19q^4 + 8q^3 \\
\end{array}$ \\
\hline
10 & \!\!\!\!\!
$\begin{array}{l}
2q^{17} + 16q^{16} + 56q^{15} + 158q^{14} + 43q^{13} - 1045q^{12} - 901q^{11} + 3558q^{10} - 536q^9 \\
 - 2793q^8 + 1065q^7 + 863q^6 - 538q^5 + 40q^4 + 12q^3 \\
\end{array}$ \\
\hline
11 & \!\!\!\!\!
$\begin{array}{l}
4q^{17} + 18q^{16} + 81q^{15} + 158q^{14} + 23q^{13} - 1424q^{12} - 641q^{11} + 4485q^{10} - 1302q^9 \\
 - 3668q^8 + 2167q^7 + 745q^6 - 765q^5 + 95q^4 + 26q^3 - q^2 - q \\
\end{array}$ \\
\hline
12 & \!\!\!\!\!
$\begin{array}{l}
q^{18} + 24q^{16} + 110q^{15} + 183q^{14} - 106q^{13} - 1771q^{12} - 82q^{11} + 5180q^{10} - 2441q^9 \\
 - 3968q^8 + 3046q^7 + 579q^6 - 928q^5 + 156q^4 + 17q^3 \\
\end{array}$\\
\hline
13 & \!\!\!\!\!
$\begin{array}{l}
8q^{17} + 18q^{16} + 94q^{15} + 248q^{14} - 244q^{13} - 1816q^{12} + 170q^{11} + 5344q^{10} - 2648q^9 \\
 - 4700q^8 + 4048q^7 + 313q^6 - 1062q^5 + 181q^4 + 59q^3 - 12q^2 - q \\
\end{array}$\\
\hline
14 & \!\!\!\!\!
$\begin{array}{l}
q^{18} + 34q^{16} + 78q^{15} + 255q^{14} - 221q^{13} - 1960q^{12} + 254q^{11} + 5666q^{10} - 2769q^9 \\ 
 - 5320q^8 + 4530q^7 + 402q^6 - 1140q^5 + 114q^4 + 88q^3 - 12q^2
\end{array}$\\
\hline
15 & \!\!\!\!\!
$\begin{array}{l}
5q^{17} + 16q^{16} + 102q^{15} + 159q^{14} - 34q^{13} - 1971q^{12} + 5874q^{10} - 2611q^9 - 5739q^8 \\
 + 4463q^7 + 1245q^6 - 1965q^5 + 389q^4 + 89q^3 - 19q^2 - 3q \\
\end{array}$\\
\hline
16 & \!\!\!\!\!
$\begin{array}{l}
2q^{17} + 19q^{16} + 68q^{15} + 177q^{14} - 34q^{13} - 1546q^{12} - 781q^{11} + 5859q^{10} - 1497q^9 \\
 - 6700q^8 + 4519q^7 + 1502q^6 - 2007q^5 + 344q^4 + 94q^3 - 17q^2 - 2q \\
\end{array}$\\
\hline
17 & \!\!\!\!\!
$\begin{array}{l}
2q^{17} + 8q^{16} + 50q^{15} + 168q^{14} + 41q^{13} - 1220q^{12} - 1275q^{11} + 5183q^{10} + 93q^9\\
 - 7338q^8 + 3805q^7 + 2323q^6 - 2191q^5 + 205q^4 +  169q^3 - 16q^2 - 8q + 1\\
\end{array}$\\
\hline
18 & \!\!\!\!\!
$\begin{array}{l}
6q^{16} + 26q^{15} + 135q^{14} + 170q^{13} - 981q^{12} - 1538q^{11} + 4160q^{10} + 1687q^9 - 7438q^8 \\
 + 2573q^7 + 3289q^6 - 2439q^5 + 199q^4 + 174q^3 - 20q^2 - 3q
\end{array}$\\
\hline
19 & \!\!\!\!\!
$\begin{array}{l}
2q^{16} + 20q^{15} + 69q^{14} + 171q^{13} - 449q^{12} - 1653q^{11} + 2529q^{10} + 3116q^9 - 6395q^8 \\
 + 648q^7 + 3964q^6 - 2151q^5 - 110q^4 + 260q^3 - 7q^2 - 15q + 1 \\
\end{array}$\\
\hline
20 & \!\!\!\!\!
$\begin{array}{l}
6q^{15} + 47q^{14} + 137q^{13} - 140q^{12} - 1274q^{11} + 720q^{10} + 3947q^9 - 4449q^8 - 1556q^7 \\
 + 4192q^6 - 1494q^5 - 358q^4 + 228q^3 - 3q^2 - 2q - 1\\
\end{array}$ \\
\hline
21 & \!\!\!\!\!
$\begin{array}{l}
q^{15} + 20q^{14} + 90q^{13} + 13q^{12} - 700q^{11} - 421q^{10} + 3192q^9 - 1535q^8 - 3435q^7 \\
 + 3743q^6 - 351q^5 - 993q^4 + 389q^3 + q^2 - 14q\\
\end{array}$ \\
\hline
22 & \!\!\!\!\!
$\begin{array}{l}
4q^{14} + 41q^{13} + 86q^{12} - 292q^{11} - 707q^{10} + 1611q^9 + 826q^8 - 3505q^7 + 1826q^6 \\
 + 988q^5 - 1152q^4 + 232q^3 + 64q^2 - 23q + 1 \\
\end{array}$ \\
\hline
23 & \!\!\!\!\!
$\begin{array}{l}
9q^{13} + 64q^{12} - 17q^{11} - 568q^{10} + 391q^9 + 1481q^8 - 1988q^7 - 260q^6 \\
 + 1665q^5 - 817q^4 - 67q^3 + 126q^2 - 17q - 2 \\
\end{array}$ \\
\hline
24 & \!\!\!\!\!
$\begin{array}{l}
16q^{12} + 74q^{11} - 243q^{10} - 237q^9 + 1061q^8 - 412q^7 - 1127q^6 \\ 
 + 1162q^5 - 111q^4 - 293q^3 + 115q^2 - 2q - 3 \\
\end{array}$ \\
\hline
25 & \!\!\!\!\!
$\begin{array}{l}
26q^{11} + 22q^{10} - 316q^9 + 344q^8 + 403q^7 - 891q^6 + 302q^5 \\
+ 327q^4 - 263q^3 + 33q^2 + 16q - 3 \\
\end{array}$ \\
\hline
26 & 
$
31q^{10} - 56q^9 - 139q^8 + 428q^7 - 260q^6 - 223q^5 + 327q^4 - 90q^3 - 37q^2 + 21q - 2
$ \\
\hline
27 & 
$
25q^9 - 84q^8 + 46q^7 + 150q^6 - 246q^5 + 105q^4 + 34q^3 - 36q^2 + 5q + 1
$ \\
\hline
28 & 
$
14q^8 - 60q^7 + 88q^6 - 30q^5 - 49q^4 + 52q^3 - 14q^2 - 2q + 1
$ \\
\hline
29 &
$
5q^7 - 24q^6 + 45q^5 - 40q^4 + 15q^3 - q
$ \\
\hline
30 &
$
 q^6 - 5q^5 + 10q^4 - 10q^3 + 5q^2 - q
$ \\
\hline
\end{longtable}

\begin{longtable}{|c|l|}
\multicolumn{2}{c}{$n=13$} \\
\hline
$e$ & \multicolumn{1}{c|}{$N_{n,e}(q)$} \\
\hline\endfirsthead
\multicolumn{2}{c}{$n=13$ (continued)} \\
\hline
$e$ & \multicolumn{1}{c|}{$N_{n,e}(q)$} \\
\hline\endhead
\multicolumn{2}{c}{\textit{Continued on the next page}}
\endfoot
\endlastfoot
0 &
$q^{12}$ \\
\hline
1 &
$
10q^{12} - 9q^{11} - q^{10}
$ \\
\hline
2 &
$
8q^{13} + 30q^{12} - 65q^{11} + 18q^{10} + 9q^9
$ \\
\hline
3 &
$
6q^{14} + 39q^{13} + 5q^{12} - 167q^{11} + 114q^{10} + 31q^9 - 28q^8
$ \\
\hline
4 &
$
4q^{15} + 30q^{14} + 67q^{13} - 103q^{12} - 275q^{11} + 358q^{10} + 20q^9 - 135q^8 + 34q^7
$ \\
\hline
5 & \!\!\!\!\!
$\begin{array}{l}
2q^{16} + 20q^{15} + 67q^{14} + 82q^{13} - 394q^{12} - 248q^{11} + 859q^{10} - 229q^9 - 298q^8 \quad \quad \quad \\
+ 148q^7 - 9q^6 \\
\end{array}$ \\
\hline
6 & \!\!\!\!\!
$\begin{array}{l}
12q^{16} + 51q^{15} + 133q^{14} - 65q^{13} - 799q^{12} + 218q^{11} + 1379q^{10} - 847q^9 \\
- 458q^8 + 454q^7 - 69q^6 - 9q^5 
\end{array}$ \\
\hline
7 & \!\!\!\!\!
$\begin{array}{l}
3q^{17} + 28q^{16} + 134q^{15} + 125q^{14} - 411q^{13} - 1151q^{12} + 1320q^{11} + 1524q^{10} \\
- 1892q^9 - 348q^8 + 964q^7 - 311q^6 + 11q^5 + 4q^4
\end{array}$ \\
\hline
8 & \!\!\!\!\!
$\begin{array}{l}
2q^{18} + 6q^{17} + 78q^{16} + 193q^{15} + 77q^{14} - 1069q^{13} - 1025q^{12} + 2846q^{11} \\
+ 948q^{10} - 3225q^9 + 375q^8 + 1398q^7 - 678q^6 + 67q^5 + 7q^4
\end{array}$ \\
\hline
9 & \!\!\!\!\!
$\begin{array}{l}
6q^{18} + 26q^{17} + 95q^{16} + 321q^{15} - 69q^{14} - 2094q^{13} - 207q^{12} + 4740q^{11} - 746q^{10} \\
- 4449q^9 + 1744q^8 + 1604q^7 - 1169q^6 + 190q^5 + 11q^4 - 3q^3
\end{array}$ \\
\hline
10 & \!\!\!\!\!
$\begin{array}{l}
2q^{19} + 5q^{18} + 59q^{17} + 170q^{16} + 345q^{15} - 445q^{14} - 2925q^{13} + 968q^{12} + 7205q^{11} \\
- 4283q^{10} - 4943q^9 + 3818q^8 + 1357q^7 - 1697q^6 + 360q^5 + 4q^4
\end{array}$\\
\hline
11 & \!\!\!\!\!
$\begin{array}{l}
3q^{19} + 20q^{18} + 70q^{17} + 257q^{16} + 361q^{15} - 1001q^{14} - 3656q^{13} + 2824q^{12} \\
+ 9253q^{11} - 8839q^{10} - 4743q^9 + 6750q^8 + 222q^7 - 2146q^6 + 619q^5 \\
+ 22q^4 - 14q^3 - 2q^2
\end{array}$\\
\hline
12 & \!\!\!\!\!
$\begin{array}{l}
2q^{19} + 34q^{18} + 137q^{17} + 277q^{16} + 255q^{15} - 1559q^{14} - 4125q^{13} + 5221q^{12} \\
+ 10401q^{11} - 13938q^{10} - 3533q^9 + 9951q^8 - 1711q^7 - 2312q^6 + 954q^5 \\
- 39q^4 - 14q^3 - q^2
\end{array}$\\
\hline
13 & \!\!\!\!\!
$\begin{array}{l}
2q^{20} + 6q^{19} + 30q^{18} + 179q^{17} + 358q^{16} + 70q^{15} - 1933q^{14} - 4873q^{13} + 8237q^{12} \\
+ 10694q^{11} - 19027q^{10} - 1647q^9 + 13190q^8 - 4302q^7 - 2097q^6 + 1202q^5 - 53q^4 \\
- 38q^3 + q^2 + q
\end{array}$\\
\hline
14 & \!\!\!\!\!
$\begin{array}{l}
2q^{20} + 10q^{19} + 39q^{18} + 169q^{17} + 531q^{16} - 77q^{15} - 2875q^{14} - 4683q^{13} \\
+ 11171q^{12} + 10159q^{11} - 23953q^{10} + 721q^9 + 16398q^8 - 7185q^7 \\
- 1638q^6 + 1262q^5 + 5q^4 - 56q^3
\end{array}$\\
\hline
15 & \!\!\!\!\!
$\begin{array}{l}
2q^{20} + 10q^{19} + 51q^{18} + 210q^{17} + 467q^{16} + 20q^{15} - 3792q^{14} - 4140q^{13} + 13647q^{12} \\
+ 8608q^{11} - 27735q^{10} + 3996q^9 + 18068q^8 - 9282q^7 - 1833q^6 + 1895q^5 - 114q^4 \\
- 82q^3 + 3q^2 + q
\end{array}$\\
\hline
16 & \!\!\!\!\!
$\begin{array}{l}
3q^{20} + 12q^{19} + 51q^{18} + 198q^{17} + 477q^{16} + 24q^{15} - 3872q^{14} - 4833q^{13} + 15815q^{12} \\
+ 8070q^{11} - 30588q^{10} + 5296q^9 + 20815q^8 - 12199q^7 - 869q^6 + 1691q^5 + 68q^4 \\
- 183q^3 + 24q^2
\end{array}$\\
\hline
17 & \!\!\!\!\!
$\begin{array}{l}
17q^{19} + 49q^{18} + 200q^{17} + 426q^{16} + 84q^{15} - 3820q^{14} - 5318q^{13} + 16673q^{12} \\
+ 8957q^{11} - 33680q^{10} + 6526q^9 + 22265q^8 - 12455q^7 - 2340q^6 + 2722q^5 \\
- 185q^4 - 123q^3 - 2q^2 + 4q
\end{array}$\\
\hline
18 & \!\!\!\!\!
$\begin{array}{l}
2q^{20} + 66q^{18} + 191q^{17} + 421q^{16} - 84q^{15} - 3371q^{14} - 4961q^{13} + 14736q^{12} \\
+ 11064q^{11} - 34008q^{10} + 4206q^9 + 25788q^8 - 14341q^7 - 2490q^6 + 3196q^5 \\
- 280q^4 - 137q^3 - 2q^2 + 4q
\end{array}$\\
\hline
19 & \!\!\!\!\!
$\begin{array}{l}
2q^{20} + 8q^{19} + 15q^{18} + 204q^{17} + 358q^{16} + 62q^{15} - 2791q^{14} - 5331q^{13} + 12413q^{12} \\
+ 13799q^{11} - 32500q^{10} - 11q^9 + 27565q^8 - 12789q^7 - 4512q^6 + 4036q^5 - 361q^4 \\
- 177q^3 + 4q^2 + 5q + 1
\end{array}$\\
\hline
20 & \!\!\!\!\!
$\begin{array}{l}
6q^{19} + 30q^{18} + 81q^{17} + 401q^{16} + 210q^{15} - 2219q^{14} - 5386q^{13} + 9186q^{12} \\
+ 16436q^{11} - 28459q^{10} - 7048q^9 + 29874q^8 - 10394q^7 - 6817q^6 + 4536q^5 \\
- 197q^4 - 256q^3 + 10q^2 + 5q + 1
\end{array}$
\\
\hline
21 & \!\!\!\!\!
$\begin{array}{l}
2q^{19} + 20q^{18} + 75q^{17} + 209q^{16} + 515q^{15} - 1553q^{14} - 5346q^{13} + 5582q^{12} \\
+ 18560q^{11} - 23145q^{10} - 13967q^9 + 30496q^8 - 6824q^7 - 9149q^6 + 4856q^5 \\
- 66q^4 - 284q^3 + 18q^2 + q
\end{array}$\\
\hline
22 & \!\!\!\!\!
$\begin{array}{l}
9q^{18} + 52q^{17} + 174q^{16} + 360q^{15} - 576q^{14} - 4821q^{13} + 2049q^{12} + 18239q^{11} \\
- 15426q^{10} - 19574q^9 + 27484q^8 - 725q^7 - 12331q^6 + 5172q^5 + 341q^4 - 455q^3 \\
+ 17q^2 + 12q - 1
\end{array}$\\
\hline
23 & \!\!\!\!\!
$\begin{array}{l}
34q^{17} + 119q^{16} + 282q^{15} - 92q^{14} - 3506q^{13} - 770q^{12} + 14907q^{11} - 5897q^{10} \\
- 22771q^9 + 21208q^8 + 4989q^7 - 12790q^6 + 3895q^5 + 843q^4 - 452q^3 \\
- 13q^2 + 14q
\end{array}$\\
\hline
24 & \!\!\!\!\!
$\begin{array}{l}
2q^{18} + 89q^{16} + 193q^{15} + 140q^{14} - 1976q^{13} - 2568q^{12} + 10081q^{11} + 2234q^{10} \\
- 21383q^9 + 11607q^8 + 10263q^7 - 11227q^6 + 1372q^5 + 1668q^4 - 448q^3 - 70q^2 \\
+ 22q + 1
\end{array}$\\
\hline
25 & \!\!\!\!\!
$\begin{array}{l}
4q^{17} + 5q^{16} + 198q^{15} + 160q^{14} - 912q^{13} - 2613q^{12} + 4698q^{11} + 7305q^{10} \\
- 16170q^9 + 1491q^8 + 13568q^7 - 8488q^6 - 715q^5 + 1784q^4 - 239q^3 \\
- 88q^2 + 12q
\end{array}$\\
\hline
26 & \!\!\!\!\!
$\begin{array}{l}
12q^{16} + 25q^{15} + 273q^{14} - 280q^{13} - 1915q^{12} + 839q^{11} + 7817q^{10} - 8424q^9 \\
- 5856q^8 + 12745q^7 - 4097q^6 - 2884q^5 + 1946q^4 - 85q^3 - 127q^2 + 8q + 3
\end{array}$\\
\hline
27 & \!\!\!\!\!
$\begin{array}{l}
22q^{15} + 60q^{14} + 225q^{13} - 1038q^{12} - 1202q^{11} + 5432q^{10} - 1029q^9 - 9263q^8 \\
+ 8865q^7 + 326q^6 - 3954q^5 + 1641q^4 - 4q^3 - 83q^2 + q + 1
\end{array}$\\
\hline
28 & \!\!\!\!\!
$\begin{array}{l}
38q^{14} + 89q^{13} - 68q^{12} - 1450q^{11} + 1845q^{10} + 3318q^9 - 7623q^8 + 2785q^7 \\
+ 4014q^6 - 3964q^5 + 806q^4 + 361q^3 - 165q^2 + 13q + 1
\end{array}$\\
\hline
29 & \!\!\!\!\!
$\begin{array}{l}
54q^{13} + 89q^{12} - 561q^{11} - 455q^{10} + 3213q^9 - 2553q^8 - 2653q^7 + 5165q^6 \\
- 2505q^5 - 167q^4 + 490q^3 - 120q^2 + 2q + 1
\end{array}$\\
\hline
30 & \!\!\!\!\!
$\begin{array}{l}
74q^{12} - 25q^{11} - 736q^{10} + 1122q^9 + 1014q^8 - 3589q^7 + 2611q^6 + 327q^5 \\
- 1341q^4 + 622q^3 - 59q^2 - 25q + 5
\end{array}$\\
\hline
31 & \!\!\!\!\!
$\begin{array}{l}
82q^{11} - 189q^{10} - 330q^9 + 1431q^8 - 1265q^7 - 675q^6 + 1961q^5 - 1360q^4 + 339q^3 \\
+ 30q^2 - 27q + 3
\end{array}
$
\\
\hline
32 & \!\!\!\!\!
$\begin{array}{l}
70q^{10} - 265q^9 + 179q^8 + 603q^7 - 1292q^6 + 878q^5 + 46q^4 - 393q^3 \\
+ 218q^2 - 47q + 3
\end{array}$\\
\hline
33 &
$
44q^9 - 211q^8 + 350q^7 - 114q^6 - 372q^5 + 546q^4 - 314q^3 + 70q^2 + 4q - 3
$\\
\hline
34 & 
$
20q^8 - 109q^7 + 236q^6 - 241q^5 + 80q^4 + 65q^3 - 76q^2 + 29q - 4
$
\\
\hline
35 & 
$
6q^7 - 35q^6 + 85q^5 - 110q^4 + 80q^3 - 31q^2 + 5q
$ \\
\hline
36 & 
$
q^6 - 6q^5 + 15q^4 - 20q^3 + 15q^2 - 6q + 1
$ \\
\hline
\end{longtable}

\end{appendix}


\begin{thebibliography}{10}
\expandafter\ifx\csname url\endcsname\relax
  \def\url#1{\texttt{#1}}\fi
\expandafter\ifx\csname urlprefix\endcsname\relax\def\urlprefix{URL }\fi
\expandafter\ifx\csname href\endcsname\relax
  \def\href#1#2{#2} \def\path#1{#1}\fi

\bibitem{MAGMA}
W.~Bosma, J.~Cannon, C.~Playoust, The {M}agma algebra system. {I}. {T}he user
  language., J.\ Symbolic Comput. 24~(3-4) (1997) 98--128.

\bibitem{Boyarchenko2006}
M.~Boyarchenko, Base change maps for unipotent algebra groups, preprint, arXiv:
  math/0601133 (2006).

\bibitem{CRI}
C.~Curtis, I.~Reiner, Methods of representation theory: with applications to
  finite groups and orders, Vol.~1, John Wiley {\&} Sons, Inc., 1981.

\bibitem{Halasi2006}
Z.~Halasi, On the characters of the unit group of {DN}-algebras, J.\ Algebra
  302 (2006) 678--685.

\bibitem{Higman1}
G.~Higman, Enumerating $p$-groups. {I}: {I}nequalities, Proc. London Math. Soc.
  (3) 10 (1960) 24--30.

\bibitem{IsaacsBook}
I.M.~Isaacs, Character theory of finite groups, Dover Publications, New York,
  1994.

\bibitem{Isaacs1995}
I.M.~Isaacs, Characters of groups associated with finite algebras, J.\ Algebra
  177 (1995) 708--730.


\bibitem{Isaacs2007}
I.M.~Isaacs, Counting characters of upper triangular groups, J.\ Algebra 315
  (2007) 698--719.

\bibitem{IK2005}
I.M.~Isaacs, D.~Karagueuzian, Involutions and characters of upper triangular
  matrix groups, Math.\ Comp. 74~(252) (2005) 2027--2033.

\bibitem{Marberg2008}
E.~Marberg, Constructing modules of algebra groups, honors Thesis, Stanford
  University, 2008.

\bibitem{Slattery1986}
M.~Slattery, Computing character degrees of $p$-groups, J.\ Symbolic Comput.
  2~(1) (1986) 51--58.

\bibitem{pol2003}
A.~Vera-L\'opez, J.~Arregi, Conjugacy classes in unitriangular matrices, Linear
  Algebra Appl. 370 (2003) 85--124.
\end{thebibliography}

\end{document}